\documentclass[11pt]{amsart}

\usepackage{amscd, amsfonts, amsmath, amssymb, amsthm, enumerate, flexisym, float, graphicx, hyperref, mathtools, pdfpages, pinlabel, subcaption, tikz, tikz-cd, url, wasysym}

\makeatletter
\@namedef{subjclassname@2020}{Mathematics Subject Classification \textup{2020}}
\makeatother

\newtheorem{theorem}{Theorem}[section]
\newtheorem{corollary}[theorem]{Corollary}
\newtheorem{lemma}[theorem]{Lemma}
\newtheorem{proposition}[theorem]{Proposition}
\newtheorem{conjecture}[theorem]{Conjecture}
\theoremstyle{definition}

\newtheorem{example}[theorem]{Example}

\newtheorem{remark}[theorem]{Remark}

\numberwithin{equation}{subsection}
\newtheorem*{ack}{Acknowledgement}

\newcommand{\Conj}{\operatorname{Conj}}
\newcommand{\Core}{\operatorname{Core}}

\newcommand{\Env}{\operatorname{Env}}

\newcommand{\C}{\operatorname{C}}

\newcommand{\id}{\operatorname{Id}}

\newcommand{\mcg}{\mathcal{M}}
\newcommand{\dq}{\mathcal{D}}

\usetikzlibrary{cd, arrows, calc, knots, patterns, positioning, shapes.callouts, shapes.geometric, shapes.symbols}
\tikzset{>=stealth', arrow/.style={->}}

\begin{document}

\title{Presentations of Dehn quandles}
\author{Neeraj K. Dhanwani}
\author{Hitesh Raundal}
\author{Mahender Singh}

\address{Department of Mathematical Sciences, Indian Institute of Science Education and Research (IISER) Mohali, Sector 81, S. A. S. Nagar, P. O. Manauli, Punjab 140306, India.}
\email{neerajk.dhanwani@gmail.com}
\email{hiteshrndl@gmail.com}
\email{mahender@iisermohali.ac.in}


\begin{abstract}
The paper gives two approaches to write explicit presentations for the class of Dehn quandles using presentations of their underlying groups. The first approach gives finite presentations for Dehn quandles of a class of Garside groups and Gaussian groups. The second approach is for general Dehn quandles when the centralisers of generators of their underlying groups are known. Several examples including Dehn quandles of spherical Artin groups, surface groups and mapping class groups of orientable surfaces are given to illustrate the results. 

\end{abstract}

\maketitle

\section{Introduction}
A quandle is an algebraic system with a binary operation that encodes the three Reidemeister moves of planar diagrams of links in the 3-space. Besides being fundamental to knot theory as observed first in \cite{Joyce1979, Joyce1982, Matveev1982}, these objects appear in the study of mapping class groups \cite{Zablow1999}, set-theoretic solutions of the quantum Yang-Baxter equation and Yetter-Drinfeld Modules \cite{Eisermann2005}, Riemannian symmetric spaces \cite{Loos1} and Hopf algebras \cite{Andruskiewitsch2003}, to name a few. Many classical topological, combinatorial and geometric knot invariants such as the knot group \cite{Joyce1979, Matveev1982}, the knot coloring, the Conway polynomial, the Alexander polynomial \cite{Joyce1979} and the volume of the complement in the 3-sphere of a hyperbolic knot \cite{MR3226796} can be retrieved from the knot quandle. Though quandles give strong invariants of knots, the isomorphism problem for them is hard. This has motivated search for newer properties, constructions and invariants of quandles themselves.
\par

Understanding of presentations of quandles is a fundamental problem and determining a presentation is challenging in general. Even for the simplest quandles arising from groups, such as conjugation quandles of  infinite abelian groups, the number of generators and relations turn out to be infinite. In this paper, we give two approaches to write explicit presentations for a fairly large class of quandles called Dehn quandles introduced recently in \cite{Dhanwani-Raundal-Singh-2021}. We believe that understanding of presentations has the potential to lead to a combinatorial quandle theory.
\par

The notion of a Dehn quandle of a group is motivated by two classes of examples. The first being the free quandle on a set $S$, which is simply the union of conjugacy classes of elements of $S$ in the free group on $S$, equipped with the quandle operation of conjugation. The second class of examples is given by surfaces. Let $\mathcal{M}_g$ be the mapping class group of a closed orientable surface $S_g$ of genus $g \ge 1$ and $\mathcal{D}_g^{ns}$ the set of isotopy classes of non-separating simple closed curves on $S_g$. It is known that $\mathcal{M}_g$ is generated by Dehn twists along finitely many simple closed curves from $\mathcal{D}_g^{ns}$ \cite[Theorem 4.1]{Farb-Margalit2012}. The binary operation $$\alpha * \beta= T_\beta(\alpha),$$ where $\alpha, \beta \in \mathcal{D}_g^{ns}$ and $T_\beta$ is the Dehn twist along $\beta$, turns $\mathcal{D}_g^{ns}$ into a quandle called the Dehn quandle of the surface $S_g$. It turns out that $\mathcal{D}_g^{ns}$ can be seen as a subquandle of the conjugation quandle $\Conj(\mathcal{M}_g)$ of the mapping class group $\mathcal{M}_g$ of $S_g$, by identifying the isotopy class of a simple closed curve with the isotopy class of the corresponding Dehn twist.  These quandles originally appeared in the work of Zablow \cite{Zablow1999, Zablow2003}. He derived a homology theory based on Dehn quandles of surfaces \cite{Zablow2008} and showed that isomorphism classes of Lefschetz fibrations over a disk correspond to quandle homology classes in dimension two. Further, in \cite{ChamanaraHuZablow}, the Dehn quandle structure of the torus has been extended to a quandle structure on the set of its measured geodesic foliations and the quandle homology of this extended quandle has been explored. On the other hand, \cite{kamadamatsumoto,Yetter2002, Yetter2003} considered a quandle structure on the set of isotopy classes of simple closed arcs on an orientable surface with at least two punctures, and called it the quandle of cords. In the case of a disk with $n$ punctures, this is simply the Dehn quandle of the braid group $B_n$ with respect to its standard Artin generators. A presentation for the quandle of cords of the real plane and the 2-sphere has been given in \cite{kamadamatsumoto}. Beyond these special cases, we have not seen explicit presentations of Dehn quandles of surfaces in the literature. In fact, other than the well-known procedure of writing presentations of link quandles of tame links via their planar diagrams, general results on presentations of quandles do not seems to be known. However, an analogue of Tietze's theorem relating two finite presentations of a quandle is known due to
Fenn and Rourke \cite[Theorem 4.2]{MR1194995}. Dehn quandles of groups with respect to their subsets include many well-known constructions of quandles from groups including conjugation quandles, free quandles, Coxeter quandles \cite{TAkita}, Dehn quandles of closed orientable surfaces, quandles of cords of orientable surfaces, knot quandles of prime knots, core quandles of groups and generalized Alexander quandles of groups with respect to fixed-point free automorphisms, to name a few. See \cite{Dhanwani-Raundal-Singh-2021} for details.
\par

The paper is organized as follows. We give two approaches to write explicit presentations of Dehn quandles using presentations of their underlying groups. In Section \ref{sec-gaussian-and-garside-quandles}, using Garside theory, we give finite presentations of Dehn quandles of groups of fractions of Garside monoids and Gaussian monoids (Theorem \ref{thm-present-dehn-quandle-right-gaussian-group} and Theorem \ref{thm-present-dehn-quandle-left-gaussian-group}). We give examples of presentations for Dehn quandles using this approach including those of spherical Artin quandles.  In Section \ref{section-present-dehn-quandles}, we prove a general result giving a presentation of the Dehn quandle of a group with respect to a generating set when the centraliser of each generator is known (Theorem \ref{presentation-dehn-quandle}). As a consequence, it follows that is $G=\langle S \mid R\rangle$ is a finitely presented group such that the centraliser of each generator from $S$ is finitely generated, then the Dehn quandle $\dq(S^G)$ is finitely presented (Corollary \ref{finitely gen dehn quandle}). Although Theorem \ref{presentation-dehn-quandle} is general, finding generating sets for centralisers of elements in interesting classes of groups like Garside groups and Artin groups is usually challenging. We give presentations for Dehn quandles of surface groups, braid groups and mapping class groups of orientable surfaces to illustrate the result. Examples of braid groups show that presentations of Garside quandles given by Theorem \ref{presentation-dehn-quandle} usually have larger number of relations than the ones given by Theorem \ref{thm-present-dehn-quandle-right-gaussian-group}.
\medskip

\section{Presentations of Gaussian quandles and Garside quandles}\label{sec-gaussian-and-garside-quandles}

We refer the reader to \cite{Joyce1979, Joyce1982, Matveev1982} for basic facts on quandles. Throughout the paper, we consider only right distributive quandles and follow conventions used in \cite{Dhanwani-Raundal-Singh-2021}. If $(X,*)$ is such a quandle with a generating set $S$, then by \cite[Lemma 4.4.8]{Winker1984}, any element of $X$ can be written in a left-associated product of the form
\begin{equation*}
\left(\left(\cdots\left(\left(a_0*^{e_1}a_1\right)*^{e_2}a_2\right)*^{e_3}\cdots\right)*^{e_{n-1}}a_{n-1}\right)*^{e_n}a_n
\end{equation*}
for $a_i \in S$ and $e_i \in \{ 1, -1\}$. For simplicity of notation, we write such an element as
\begin{equation*}
a_0*^{e_1}a_1*^{e_2}\cdots*^{e_n}a_n.
\end{equation*}
\par

Let $G$ be a group, $S$ a non-empty subset of $G$ and $S^G$ the set of all conjugates of elements of $S$ in $G$. The {\it Dehn quandle} $\mathcal{D}(S^G)$ of $G$ with respect to $S$ is the set $S^G$ equipped with the binary operation of conjugation, that is, $$x*y=yxy^{-1}$$ for $x, y \in \mathcal{D}(S^G)$. We refer to our recent work \cite{Dhanwani-Raundal-Singh-2021} for examples and basic results on Dehn quandles. 
\medskip

Understanding of presentations of quandles is fundamental to the development of quandle theory.   In this section, we give presentations of Dehn quandles of certain Gaussian groups and Garside groups, which we shall refer as {\it Gaussian quandles} and {\it Garside quandles}, respectively. The notion of a Gaussian group and a Garside group was first introduced by Dehornoy and Paris \cite{Dehornoy-Paris-1999} and developed further in the works of Dehornoy \cite{Dehornoy2002} and Picantin \cite{Picantin2000,Picantin2001a,Picantin2001b}. The definition of a Garside group which many authors use first appeared in \cite{Dehornoy2002}. Many equivalent definitions of a Garside group can be seen in the literature. For example, see \cite{Birman-Gebhardt-Meneses-2007,Crisp-Paris-2005,Franco-Meneses-2003,Gebhardt-2005}. Note that a Garside group is referred as a {\it small Gaussian group} or a {\it thin Gaussian group} in \cite{Dehornoy-Paris-1999,Picantin2000,Picantin2001a,Picantin2001b}, and the notion of a Garside group in \cite{Dehornoy-Paris-1999} is a special kind of a Garside group (see the comment after the definition of a Garside monoid). The theory of Garside groups is largely inspired by the work of Garside \cite{Garside1969} in which he treated braid groups, and by the work of Brieskorn and Saito \cite{Brieskorn-Saito-1972} generalizing the work of Garside to Artin groups of spherical type.
\medskip

\subsection{Preliminaries on Gaussian and Garside monoids, and groups}

An element $a$ of a monoid $M$ is called an {\it atom} if $a\neq1$ and $a$ is indecomposable (i.e. if $a=bc$, then either $b=1$ or $c=1$). A monoid $M$ is called {\it atomic} if the atoms of $M$ are finite in number, they generate $M$, and if for every $x\in M$, there exists an integer $N_x$ such that $x$ cannot be written as a product of more than $N_x$ atoms. Using \cite[Proposition 2.1(iii)]{Dehornoy-Paris-1999}, we can say that every atomic monoid is {\it conical}, i.e. there are no invertible elements other than the identity element of the monoid. In a monoid $M$, an element $x$ is a {\it left divisor} of $y$, or $y$ is a {\it right multiple} of $x$, if there exists an element $z$ satisfying $y= xz$. {\it Right divisors} and {\it left multiples} are defined analogously. We denote by $x\leq_L y$ if $x$ is a left divisor of $y$, and $x\leq_R y$ if $x$ is a right divisor of $y$. In an atomic monoid $M$, the left and right divisibility relations (i.e. the relations $\leq_L$ and $\leq_R$) are respectively left and right invariant partial orders on $M$ (see \cite[Proposition 2.3]{Dehornoy-Paris-1999}). 

A {\it right Gaussian monoid} is a monoid $M$ satisfying the following properties:
\begin{itemize}
\item $M$ is atomic;
\item $M$ is left and right cancellative;
\item $(M,\leq_L)$ is a lattice (i.e. the left g.c.d. and the right l.c.m. exist and they are unique for every $x,y\in M$, see the definitions below).
\end{itemize}

A {\it left Gaussian monoid} is defined similarly. A {\it Gaussian monoid} is both left as well as right Gaussian monoid. 
Note that the preceding definition of a right Gaussian monoid is different from that of \cite{Dehornoy-Paris-1999}. We have placed an extra assumption of right cancellativity to insure that such a monoid embeds in its group of fractions (see \cite[Theorem 1.23]{Clifford-Preston-1961} for Ore's criterion). 
\par

For elements $x$ and $y$ in a right Gaussian monoid, the {\it left greatest common divisor (left g.c.d.)} of $x$ and $y$ is a common left divisor $z$ of $x$ and $y$ such that every common left divisor of $x$ and $y$ is a left divisor of $z$. The {\it right least common multiple (right l.c.m.)} of $x$ and $y$ is a common right multiple $z$ of $x$ and $y$ such that every common right multiple of $x$ and $y$ is a right multiple of $z$. Similarly, in a left Gaussian monoid, we can define the {\it right g.c.d.} and the {\it left l.c.m.} of elements.
\par

For elements $x$ and $y$ in a right Gaussian monoid, denote the left g.c.d. and the right l.c.m. of $x$ and $y$ respectively by $x\wedge y$ and $x\vee y$. The {\it right residue} of $x$ in $y$ (denoted by $x\backslash y$) is the unique element $z$ satisfying $x\vee y=xz$. Thus, we have
\begin{equation*}
x\vee y=x(x\backslash y)=y(y\backslash x)= y \vee x.
\end{equation*}
\par

A {\it Garside element} of a monoid $M$ is an element $\Delta\in M$ such that the left divisors of $\Delta$ coincide with the right divisors of $\Delta$, they are finite in number, and they generate $M$. We say that a Gaussian monoid $M$ is a {\it Garside monoid} if it contains a Garside element. 
By \cite[Proposition 2.2]{Dehornoy-Paris-1999}, the set of atoms in $M$ is contained in every generating subset of $M$. It follows that atoms are also divisors of $\Delta$. In \cite{Dehornoy-Paris-1999,Picantin2000,Picantin2001a,Picantin2001b}, a Garside monoid is referred as a {\it small Gaussian monoid} or a {\it thin Gaussian monoid}. One can look at \cite[Proposition 2.1]{Dehornoy2002} for a necessary and sufficient condition on a monoid to be a Garside monoid. Note that the notion of a {\it Garside monoid} in \cite{Dehornoy-Paris-1999} is slightly restricted in the sense that the left l.c.m. and the right l.c.m. of atoms in a Garside monoid should coincide and it is a Garside element in the sense defined above.
\par

The {\it group of fractions} of a monoid $M$ is the group which has the same presentation as that of the monoid $M$. In other words, if a monoid $M=\langle S \mid R \rangle$ is the quotient of the free monoid on $S$ modulo the relations in $R$, then its group of fractions is the quotient of the free group on $S$ modulo the relations in $R$. 
\par

A {\it Gaussian group} (respectively, a {\it Garside group}) is the group of fractions of a Gaussian monoid (respectively, of a Garside monoid). Similarly, we can also define a {\it right Gaussian group} and a {\it left Gaussian group}.
\par
An {\it Ore monoid} is one that embeds in its group of fractions.  Ore's criterion says that if a monoid $M$ is left and right cancellative, and if any two elements of $M$ have a common right multiple, then $M$ embeds in its group of fractions \cite[Theorem 1.23]{Clifford-Preston-1961}. Note that a Gaussian monoid (respectively, a Garside monoid) satisfies Ore’s conditions, and thus embeds in the corresponding Gaussian group (respectively, in the corresponding Garside group). The same is true for right Gaussian monoids as well as left Gaussian monoids.
\par

Let $M$ be a right Gaussian monoid, $A$ the set of atoms in $M$ and $G$ the group of fractions of $M$. Then the Dehn quandle $\mathcal{D}(A^G)$ will be referred as a {\it right Gaussian quandle}. The terms {\it left Gaussian quandles}, {\it Gaussian quandles} and {\it Garside quandles} are defined similarly. For example, in case of a spherical Artin group, the Artin quandle as defined in \cite[Section 3]{Dhanwani-Raundal-Singh-2021} is a Garside quandle. Here, note that, by \cite{Brieskorn-Saito-1972} and \cite[Example 1]{Dehornoy-Paris-1999}, a spherical Artin group is a Garside group.
\par
Let $S$ be any finite set. We set the following notations:
\begin{itemize}
\item $S^{-1}$ - a set which is in one to one correspondence with $S$ (for each $s\in S$, the corresponding element in $S^{-1}$ is denoted by $s^{-1}$ which denotes the inverse of $s$ in the free group on $S$).
\item $S^*$ - the set of words on $S$ together with the empty word (i.e. it is nothing but the free monoid on $S$ and the empty word corresponds the identity element of the free monoid).
\end{itemize}
Note that, by the notation itself, the set $(S\cup S^{-1})^*$ is nothing but the free monoid on $S\cup S^{-1}$. The empty word represents the identity element in $S^*$ (or, in $(S\cup S^{-1})^*$) and we denote it by $\epsilon$ or by $1$, as per convenience of notation. For a word $x=s_1^{\delta_1}s_2^{\delta_2}\cdots s_n^{\delta_n}$ in $(S\cup S^{-1})^*$, we denote the word $s_n^{-\delta_n}s_{n-1}^{-\delta_{n-1}}\cdots s_1^{-\delta_1}$ by $x^{-1}$, where $s_i\in S$ and $\delta_i\in\{1,-1\}$. Note that, for any $x\in (S\cup S^{-1})^*$, the elements $xx^{-1}$ and $x^{-1}x$ are different elements of $(S\cup S^{-1})^*$, and they differ from the empty word. A word $x\in(S\cup S^{-1})^*$ is called {\it positive} if $x\in S^*$. For a positive word $x\in S^*$, the {\it length} of $x$ (denoted by $\ell(x)$) is the number of letters in $x$ (i.e. $\ell(x)=n$ if $x=s_1s_2\cdots s_n$, where $s_i\in S$). Note that the length of a positive word is zero if and only if it is the empty word. A word $w\in(S\cup S^{-1})^*$ is called {\it symmetric} if it is of the form $x^{-1}s^\delta x$ for some $x\in(S\cup S^{-1})^*$, $s\in S$ and $\delta\in\{-1,1\}$. The word $w$ is called {\it positive symmetric} if $\delta=1$ and it is called {\it negative symmetric} if $\delta=-1$. We denote the set of positive symmetric words on $S\cup S^{-1}$ by $S^{(S\cup S^{-1})^*}$, i.e., $$S^{(S\cup S^{-1})^*}=\{x^{-1}s x\mid s\in S\;\text{and}\;x\in(S\cup S^{-1})^*\}.$$ 
\par
Let $M$ be a monoid generated by a finite set $S$ and $G$ the group of fractions of $M$. For words $x$ and $y$ in $(S\cup S^{-1})^*$, we write $x=_G y$ if $x$ and $y$ represent the same element in $G$. Similarly, we write $x=_M y$ if $x$ and $y$ are positive words and they represent the same element in $M$. Note that if $M$ is an Ore monoid, then two positive words $x$ and $y$ are equivalent in $M$ if and only if they are equivalent in $G$. In particular, this is true for right Gaussian monoids, left Gaussian monoids and Garside monoids, since such monoids are Ore monoids.
\par
For a finite set $S$, a {\it complement} on $S$ is a map $f:S\times S\to S^*$ such that $f(s,t)$ is the empty word if and only if $s=t$. A complement $f$ on $S$ is said to be {\it homogeneous} if $\ell(f(s,t))=\ell(f(t,s))$ for every $(s,t)\in S\times S$. A presentation of a monoid (or, of a group) is called a {\it complemented presentation} if it is of the form
$$\left\langle S\mid sf(s,t)=tf(t,s) \quad \text{for}\;s,t\in S\right\rangle \qquad\text{or}\quad\left\langle S\mid g(s,t)t=g(t,s)s \quad \text{for}\;s,t\in S\right\rangle,$$ where $f$ and $g$ are complements on $S$. A complemented presentation is called {\it homogeneous} if the associated complement is homogeneous. In general, a finite presentation $\left\langle S\mid R\right\rangle $ of a monoid is {\it homogeneous} if $\ell(x)=\ell(y)$ for every relation $x=y$ in $R$, where $x,y\in S^*$. A monoid is {\it homogeneous} if it has a finite homogeneous presentation. For example, Artin monoids (equipped with Artin presentations) are homogeneous.
\par
Let $M$ be a right Gaussian monoid (or, a Garside monoid) and $S$ a finite generating set for $M$. A {\it right l.c.m. selector} on $S$ in $M$ is a complement $f$ on $S$ such that $f(s,t)$ represents the element $s\backslash t$ in $M$ for every $(s,t)\in S\times S$. By \cite[Theorem 4.1]{Dehornoy-Paris-1999}, $M$ has the complemented presentation $$\left\langle S\mid sf(s,t)=tf(t,s) \quad \text{for}\;s,t\in S\right\rangle,$$ where $f$ is the right l.c.m. selector on $S$ in $M$. The corresponding right Gaussian group (respectively, the corresponding Garside group) is the quotient of the free monoid $(S\cup S^{-1})^*$ by relations $ss^{-1}=1$, $s^{-1}s=1$ and $sf(s,t)=tf(t,s)$ for $s,t\in S$.
\par
Let $M$ be a left Gaussian monoid (or, a Garside monoid) and $S$ a finite generating set for $M$. A {\it left l.c.m. selector} on $S$ in $M$ is a complement $g$ on $S$ such that $g(s,t)$ represents the element $s/t$ in $M$ for every $(s,t)\in S\times S$. The monoid $M$ has the complemented presentation $$\left\langle S\mid g(s,t)t=g(t,s)s \quad \text{for}\;s,t\in S\right\rangle,$$ where $g$ is the left l.c.m. selector on $S$ in $M$. The corresponding left Gaussian group (respectively, the corresponding Garside group) has the same complemented presentation as given above.
\par
One can look at \cite[sections 3 and 4]{Dehornoy-Paris-1999}, for necessary and sufficient conditions on a monoid with complimented presentations to be a Gaussian monoid (or, a right Gaussian monoid, or a left Gaussian monoid). Also, see \cite[Criterion 5.9]{Crisp-Paris-2005}, \cite[propositions 5.4, 5.12 and 6.14]{Dehornoy2002} and \cite[Corollary 3.11 and Theorem 4.2]{Dehornoy-Paris-1999}, for necessary and sufficient conditions on a monoid with complimented presentations to be a Garside monoid.

\subsection{Presentations of right Gaussian quandles and Garside quandles of certain types}

For this subsection, we set the following notations:
\begin{itemize}
\item $M$ - a right Gaussian monoid;
\item $S$ - a finite generating set for $M$;
\item $A$ - the set of atoms in $M$;
\item $N$ - a Garside monoid;
\item $T$ - a finite generating set for $N$;
\item $B$ - the set of atoms in $N$;
\item $\Delta$ - a Garside element in $N$;
\item $(M,S)$, $(M,A)$, $(N,T)$, $(N,B)$, $(N,\Delta)$ and $(N,T,\Delta)$ - pairs and a triple of objects with the meaning above.
\end{itemize}

For a pair $(M,S)$, we assume throughout this subsection that elements in $S$ are pairwise distinct in $M$. The same is assumed in case of a pair $(N,T)$.
\par

For any pair $(N,\Delta)$, by \cite[lemmas 2.2 and 2.3]{Dehornoy2002}, the map $$x\mapsto(x\backslash\Delta)\backslash\Delta$$ is a permutation of the set of divisors of $\Delta$ and it extends to an automorphism of $N$, which further extends to an automorphism $\phi$ of the Garside group $H$ corresponding to $N$.

\begin{remark}\label{rmk1}
By \cite[Lemma 2.2]{Dehornoy2002}, we have $x\Delta=\Delta \phi(x)$ for all $x\in H$. Since $\phi(N)=N$ and $\phi(B)=B$, we have $N\Delta=\Delta N$ and $B\Delta=\Delta B$. 
\end{remark}

Consider the following conditions on a pair $(M,S)$:
\begin{enumerate}[(i)]
\item $(s\backslash t)s\in S$ whenever $(s,t)\in S\times S$ and $s\leq_L t$.\label{condition1}
\item $s\backslash t\leq_L s\vee t$ for all $(s,t)\in S\times S$.\label{condition2}
\item $(s\backslash t)\backslash(s\vee t)\in S$ for all $(s,t)\in S\times S$.\label{condition3}
\item $M$ has a finite homogeneous presentation $\left\langle S\mid R\right\rangle$.\label{condition4}
\item For each $s\in S$ and each $x\in M$, there exists $y=y(s,x)$ in $M$ such that $xy\leq_L sxy$.\label{condition5}
\end{enumerate}

Recall that the set of atoms in $M$ is contained in every generating subset of $M$. Thus, if the set of atoms does not satisfy condition (\ref{condition2}), then there does not exist any generating set for $M$ satisfying condition (\ref{condition2}).

\begin{remark}\label{rmk2}
Note that, in an atomic monoid, an atom $x$ is a divisor of an atom $y$ if and only if $x=y$. Thus, a pair $(M,A)$ of a right Gaussian monoid $M$ and the set $A$ of atoms in $M$ satisfies  condition (\ref{condition1}) trivially.
\end{remark}

\begin{lemma}\label{lem-conditions-ii-and-iv}
If a pair $(M,S)$ satisfies conditions (\ref{condition2}) and (\ref{condition4}), then it satisfies conditions (\ref{condition1}) and (\ref{condition3}). Moreover, in this case $S$ must be the set of atoms in $M$.
\end{lemma}

\begin{proof}
Since $(M,S)$ satisfies condition (\ref{condition4}), we have $\ell(x)=\ell(y)$ for any words $x,y\in S^*$ with $x=_M y$ and $\ell(z)=\ell(x)+\ell(y)$ for any words $x,y,z\in S^*$ with $z =_M xy$. Also, it is easy to see that $S$ must be the set of atoms in $M$. By Remark \ref{rmk2}, $(M,S)$ satisfies condition (\ref{condition1}). Since $S$ satisfies condition (\ref{condition2}), we have $$(s\backslash t)((s\backslash t)\backslash(s\vee t))=s\vee t=s(s\backslash t)$$ for all $s,t\in S$. Let $s$ and $t$ be any elements in $S$. Suppose $x$ and $y$ be positive words representing $s\backslash t$ and $(s\backslash t)\backslash(s\vee t)$, respectively. Then $xy =_M sx$ and hence $\ell(x)+\ell(y)=\ell(s)+\ell(x)$. This implies that $\ell(y)=\ell(s)=1$. Let $y=s_1s_2\cdots s_n$ for some $s_i\in S$. Since $\ell(y)=1$, we must have $n=1$. Thus, $y\in S$ and consequently $(s\backslash t)\backslash(s\vee t) \in S$. Hence,  $S$ satisfies condition \eqref{condition3}.
\end{proof}

We define the following terms:

\begin{itemize}
\item A pair $(M,S)$ is of
\begin{itemize}
\item {\it type $\mathcal{R}_1$} if it satisfies conditions (\ref{condition1}), (\ref{condition2}), (\ref{condition3}) and (\ref{condition5}).
\item {\it type $\mathcal{R}_2$} if it satisfies conditions (\ref{condition1}), (\ref{condition2}) and (\ref{condition3}), and if there exist a triple $(N,T,\Delta)$ with $T\Delta=\Delta T$ and an epimorphism $\pi:(N,T)\twoheadrightarrow (M,S)$.
\end{itemize}

\item A pair $(N,T)$ is of {\it type $\mathcal{R}_3$} if it satisfies conditions (\ref{condition1}), (\ref{condition2}) and (\ref{condition3}), and there exists a Garside element $\Delta\in N$ such that $T\Delta=\Delta T$.

\item A pair $(M,A)$ is of
\begin{itemize}
\item {\it type $\mathcal{R}_4$} if it satisfies conditions (\ref{condition2}), (\ref{condition3}) and (\ref{condition5}).
\item {\it type $\mathcal{R}_5$} if it satisfies conditions (\ref{condition2}), (\ref{condition4}) and (\ref{condition5}).
\item {\it type $\mathcal{R}_6$} if it satisfies conditions (\ref{condition2}) and (\ref{condition3}), and if there exist a triple $(N,T,\Delta)$ with $T\Delta=\Delta T$ and an epimorphism $\pi:(N,T)\twoheadrightarrow (M,A)$.
\item {\it type $\mathcal{R}_7$} if it satisfies conditions (\ref{condition2}) and (\ref{condition4}), and if there exist a triple $(N,T,\Delta)$ with $T\Delta=\Delta T$ and an epimorphism $\pi:(N,T)\twoheadrightarrow (M,A)$.
\end{itemize}

\item A pair $(N,B)$ is of
\begin{itemize}
\item {\it type $\mathcal{R}_8$} if it satisfies conditions (\ref{condition2}) and (\ref{condition3}).
\item {\it type $\mathcal{R}_9$} if it satisfies conditions (\ref{condition2}) and (\ref{condition4}).
\end{itemize}
\end{itemize}

For $i=4,5,6,7$, we say a right Gaussian monoid $M$ is of {\it type $\mathcal{R}_i$} if the pair $(M,A)$ is of type $\mathcal{R}_i$, where $A$ is the set of atoms in $M$. In this case, we also say the right Gaussian group $G$ corresponding to $M$ and the right Gaussian quandle $\mathcal{D}(A^G)$ are of {\it type $\mathcal{R}_i$}. Similarly, for $i=8,9$, we say a Garside monoid $N$ is of type {\it type $\mathcal{R}_i$} if the pair $(N,B)$ is of type $\mathcal{R}_i$, where $B$ is the set of atoms in $N$. In this case, we also say the Garside group $H$ corresponding to $N$ and the Garside quandle $\mathcal{D}(B^H)$ are of {\it type $\mathcal{R}_i$}.

\medskip

Many classes of Garside monoids are of type $\mathcal{R}_8$ and $\mathcal{R}_9$. For example, braid monoids, and more generally, Artin monoids of spherical type are Garside monoids of type $\mathcal{R}_9$. Nnote that, by \cite{Brieskorn-Saito-1972, Garside1969} and \cite[Example 1]{Dehornoy-Paris-1999}, such monoids are Garside monoids. The next two propositions give a machinery to produce families of Garside monoids of types $\mathcal{R}_8$ and $\mathcal{R}_9$.

\begin{proposition}\cite[Proposition 5.2]{Dehornoy-Paris-1999}\label{prop-dehornoy-paris-1}
Consider a finite set $S=\{x_1,x_2,\ldots,x_n\}$, $n$ positive words $u_1,u_2,\ldots,u_n$ in $S^*$, and a permutation $\delta$ of $\{1,2,\ldots,n\}$. We assume that:
\begin{enumerate}[(1)]
\item There exists a map $\nu:S \to \mathbb{Z}_{\ge 0}$ which when extended to $S^*$ by setting $\nu(\epsilon)=0$ and $\nu(uv)=\nu(u)+\nu(v)$, satisfies $$\nu\!\left(x_1u_1x_{\delta(1)}\right)=\nu\!\left(x_2u_2x_{\delta(2)}\right)=\cdots=\nu\!\left(x_nu_nx_{\delta(n)}\right).$$
\item For each index $k$, there exists an index $j$ satisfying $$x_ku_k=u_jx_{\delta(j)}.$$
\end{enumerate}
Let $M$ be the monoid defined by the presentation $$\left\langle x_1,x_2,\ldots,x_n~|~x_1u_1x_{\delta(1)}=x_2u_2x_{\delta(2)}=\cdots=x_nu_nx_{\delta(n)}\right\rangle.$$ Then $M$ is a Garside monoid.
\end{proposition}

In Proposition \ref{prop-dehornoy-paris-1}, one can check that the map $k\mapsto j$ is a permutation of $\{1,2,\ldots,n\}$. It follows from the proof of the proposition that $S$ is the set of atoms in $M$. If the words $x_iu_ix_{\delta(i)}$ and $x_ju_jx_{\delta(j)}$ both represent a right l.c.m. of $x_i$ and $x_j$ for all $i$ and $j$, then the set $S$ satisfies conditions (\ref{condition2}) and (\ref{condition3}). Consequently, $M$ is a Garside monoid of type $\mathcal{R}_8$.

\begin{example}\cite[Example 4]{Dehornoy-Paris-1999}\label{example1}
Consider a set $S=\{x_1,x_2,\ldots,x_n\}$ and  integers $p_1,p_2,\ldots,p_n$ strictly greater than $1$. Take $\delta$ to be the identity permutation and $u_i=x_i^{p_i-2}$ for each $i$. Choose $n$ positive integers $k_1, \ldots, k_n$ satisfying $k_1p_1= \cdots =k_n p_n$
 and set $\nu(x_{i_1}x_{i_2} \cdots x_{i_r})= k_{i_1}+ \cdots + k_{i_r}$. Then, by Proposition \ref{prop-dehornoy-paris-1}, the monoid $$\left\langle x_1,x_2,\ldots,x_n~|~x_1^{p_1}=x_2^{p_2}\cdots=x_n^{p_n}\right\rangle$$ is a Garside monoid. Moreover, it is of type $\mathcal{R}_8$. Note that monoids with presentations $\langle x, y \mid x^p=y^q\rangle$ have torus knot groups as their groups of fractions.
\end{example}

\begin{example}\cite[Example 5]{Dehornoy-Paris-1999}\label{example2}
For $p$ letters $x_1,x_2,\ldots,x_p$ and a positive integer $m \ge 2$, let $\text{prod}(x_1,x_2,\ldots,x_p;m)$ denotes the word
$$\underbrace{x_1x_2\cdots x_px_1x_2\cdots}_{m\;\text{terms}}.$$ 

Now, let $S=\{x_1,x_2,\ldots,x_n\}$, where $n \ge 2$. By Proposition \ref{prop-dehornoy-paris-1}, the monoid 
\begin{small}
$$\left\langle x_1,x_2,\ldots,x_n~\mid~\text{prod}(x_1,x_2,\ldots,x_n;m)=\text{prod}(x_2,x_3,\ldots, x_n, x_{1};m)=\cdots=\text{prod}(x_n,x_1,x_2,\ldots,x_{n-1};m)\right\rangle$$
\end{small}
 is a Garside monoid. Further, it is of type $\mathcal{R}_9$.
\par

As a special case of this example, the monoid
$$\langle x_1, x_2,\ldots, x_n \mid x_1 x_2 \cdots x_n=x_2x_3 \cdots x_n x_1 = \cdots =  x_n x_1 \cdots x_{n-1} \rangle$$
has the group of fractions as the fundamental group of the complement of $n$ lines through the origin in $\mathbb{C}^2$ \cite{Randell}. And, the monoid
$$\langle x_1,x_2,\ldots,x_n \mid x_1 x_2=  x_2 x_3= \cdots=  x_n x_1 \rangle$$
has the Artin group of type $I_2(n)$ as its group of fractions.

\end{example}

\begin{proposition}\cite[Proposition 5.3]{Dehornoy-Paris-1999}\label{prop-dehornoy-paris-2}
Consider Garside monoids $M_1,M_2,\ldots,M_n$ and positive integers $p_1,p_2,\ldots,p_n$ for $n \ge 2$. Let $\Delta_i$ denote the minimal Garside element of $M_i$. For each $i$, we assume that:
\begin{enumerate}[(1)]
\item There is a map $\nu_i:M_i\to\mathbb{Z}_{\ge 0}$ satisfying $\nu_i(a)>0$ for all $a\in M_i$ with $a\neq 1$ and $\nu_i(ab)=\nu_i(a)+\nu_i(b)$ for all $a,b\in M_i$.
\item If $M_i$ has only one atom, that is, if $M_i$ is isomorphic to $\mathbb{Z}^+$, then $p_i\geq 2$.
\end{enumerate}
Let $M$ be the quotient of the free product $M_1 \hexstar M_2 \hexstar \cdots \hexstar M_n$ of monoids modulo the congruence $\equiv$ generated by  $\Delta_i^{p_i}= \Delta_j^{p_j}$ for all $i, j$, that is, $$M=(M_1\hexstar M_2  \hexstar \cdots \hexstar M_n)/\equiv.$$ Then $M$ is a Garside monoid.
\end{proposition}

It is easy to see that if the Garside monoids $M_1,M_2,\ldots,M_n$ in Proposition \ref{prop-dehornoy-paris-2} are of type $\mathcal{R}_8$, and if $S_i$ is the set of atoms in $M_i$ satisfying conditions (\ref{condition2}) and (\ref{condition3}), then the monoid $M=(M_1 \hexstar M_2 \hexstar \cdots \hexstar M_n)/\equiv$ is a Garside monoid of type $\mathcal{R}_8$ with the set of atoms $S=S_1\sqcup S_2\sqcup\cdots\sqcup S_n$ satisfying conditions (\ref{condition2}) and (\ref{condition3}). It follows from the proof of the proposition that $S$ is, in fact, the set of atoms in $M$.

\begin{example}\cite[Example 6]{Dehornoy-Paris-1999}\label{example3}
Mixing presentations in examples \ref{example1} and \ref{example2} gives new examples. For example, $$\left\langle x_1,x_2,y_1,y_2,y_3~\mid~x_1^2=x_2^5=y_1y_2y_3y_1=y_2y_3y_1y_2=y_3y_1y_2y_3\right\rangle$$ is a Garside monoid by Proposition \ref{prop-dehornoy-paris-2}. Further, it is of type $\mathcal{R}_8$. 
\end{example}

\begin{example}\cite[Example 7]{Dehornoy-Paris-1999}\label{example4}
Applying Proposition \ref{prop-dehornoy-paris-2} to Artin monoids of type $B_3$ and $A_3$ shows that the monoid
\begin{align*}
&\Bigg\langle x_1,x_2,x_3,y_1,y_2,y_3~\mid~x_1x_2x_1x_2=x_2x_1x_2x_1, \quad x_1x_3=x_3x_1, \quad x_2x_3x_2=x_3x_2x_3,\\
&\hskip3mm y_1y_2y_1=y_2y_1y_2, \quad y_1y_3=y_3y_1, \quad y_2y_3y_2=y_3y_2y_3, \quad (x_1x_2x_3)^6=(y_1y_2y_3y_1y_2y_1)^3\Bigg\rangle
\end{align*}
is a Garside monoid of type $\mathcal{R}_9$.
\end{example}

There are many examples of Garside monoids in the literature, for example, \cite{Birman-Gebhardt-Meneses-2007, Dehornoy-Paris-1999, Franco-Meneses-2003, Garside1969, Gebhardt-2005, Picantin2000}. One can choose Garside monoids of type $\mathcal{R}_8$ and/or of type $\mathcal{R}_9$ from these examples. Then using Proposition \ref{prop-dehornoy-paris-2} one can produce more Garside monoids of type $\mathcal{R}_8$ and/or of type $\mathcal{R}_9$. In \cite{Picantin2000}, the cross product of monoids has been defined, and it has been proved that the cross product of Garside monoids is a Garside monoid. This allows us to construct more Garside monoids of type $\mathcal{R}_8$ and/or of type $\mathcal{R}_9$ once we have some families of such monoids.

\begin{proposition}\label{prop-pairs-of-types-r}
Pairs of types $\mathcal{R}_2$ through $\mathcal{R}_9$ are of type $\mathcal{R}_1$.
\end{proposition}

\begin{proof}
For a right Gaussian monoid $M$, by \cite[Proposition 2.2]{Dehornoy-Paris-1999}, the set of atoms in $M$ is a finite generating set for $M$. With this in mind, together with Remark \ref{rmk1}, Remark \ref{rmk2} and Lemma \ref{lem-conditions-ii-and-iv}, we can conclude that:
\begin{itemize}
\item Pairs of types $\mathcal{R}_4$ and $\mathcal{R}_5$ are of type $\mathcal{R}_1$.
\item Pairs of types $\mathcal{R}_6$ and $\mathcal{R}_7$ are of type $\mathcal{R}_2$.
\item Pairs of types $\mathcal{R}_8$ and $\mathcal{R}_9$ are of type $\mathcal{R}_3$.
\end{itemize} 
Since every Garside monoid is of course a right Gaussian monoid, a pair $(N,T)$ can be viewed as a pair $(M,S)$ of a right Gaussian monoid with $M=N$ and a finite generating set $S=T$ for $M$. Thus, by taking $M=N$, $S=T$, $\pi=\id_{N}$ (where $\id_{N}$ denotes the identity map of $N$), we conclude that a pair of type $\mathcal{R}_3$ is of type $\mathcal{R}_2$. Now, it is enough to prove that a pair of type $\mathcal{R}_2$ is of type $\mathcal{R}_1$.

Let $(M,S)$ be a pair of type $\mathcal{R}_2$. Then, there exist a triple $(N,T,\Delta)$ with $T\Delta=\Delta T$ and an epimorphism $\pi:(N,T)\twoheadrightarrow (M,S)$. In order to prove that the pair $(M,S)$ is of type $\mathcal{R}_1$, we only have to show that it satisfies condition \eqref{condition5}. Let $s\in S$ and $x\in M$ be any elements. Since $\pi$ is surjective and it maps $T$ onto $S$, there exist $s^\prime\in T$ and $x^\prime\in N$ such that $\pi(s^\prime)=s$ and $\pi(x^\prime)=x$. Since divisors of $\Delta$ generate $N$, we can write $x^\prime=d_1d_2\cdots d_n$ for some divisors $d_1,d_2,\ldots,d_n$ of $\Delta$. Here, note that $d_i$'s are not necessarily pairwise distinct. Recall that the map $\phi$ 
given by $x\mapsto (x\backslash \Delta)\backslash \Delta$ is a permutation of the set of divisors of $\Delta$.  It is easy to see that $d\Delta=\Delta\phi(d)$ for a divisor $d$ of $\Delta$. Since $\phi^{n-i}(d_i)$ is a divisor (in particular, a left divisor) of $\Delta$, we have $\phi^{n-i}(d_i)\left(\phi^{n-i}(d_i)\backslash \Delta\right) =\Delta$. Let us set $e_i=\phi^{n-i}(d_i)\backslash \Delta$ for $i=1,2,\ldots,n$. Then, for each $i=1,2,\ldots,n$, we have
\begingroup\allowdisplaybreaks
\begin{align*}
d_i\Delta^{n-i}e_i&=\Delta^{n-i}\phi^{n-i}(d_i)e_i&(\text{since}\;d_i\Delta=\Delta\phi(d_i))\\
&=\Delta^{n-i}\phi^{n-i}(d_i)\left(\phi^{n-i}(d_i)\backslash \Delta\right)&(\text{since}\;e_i=\phi^{n-i}(d_i)\backslash \Delta)\\
&=\Delta^{n-i+1}&(\text{since}\;\phi^{n-i}(d_i)\left(\phi^{n-i}(d_i)\backslash \Delta\right) =\Delta).
\end{align*}
\endgroup
Let $y^\prime=e_ne_{n-1}\cdots e_1$. Then, using $\Delta^{n-i+1}=d_i\Delta^{n-i}e_i$ once for each $i=1,2, \ldots,n$, we get
\begin{eqnarray*}
\Delta^n&=&d_1\Delta^{n-1}e_1\\
&=& d_1d_2\Delta^{n-2}e_2e_1\\
&& \vdots\\
&=& d_1d_2\cdots d_n\Delta^0e_ne_{n-1}\cdots e_1\\
&=& x^\prime y^\prime.
\end{eqnarray*}
We see that
\begingroup\allowdisplaybreaks
\begin{align*}
s^\prime x^\prime y^\prime&=s^\prime \Delta^n&(\text{since}\; x^\prime y^\prime=\Delta^n)\\
&=\Delta^n t^\prime\quad\text{for some}\;t^\prime\in T&(\text{since}\;T\Delta=\Delta T)\\
&=x^\prime y^\prime t^\prime&(\text{since}\; \Delta^n=x^\prime y^\prime).
\end{align*}
\endgroup
Let $y=\pi(y^\prime)$ and $t=\pi(t^\prime)$. Then
\begingroup\allowdisplaybreaks
\begin{align*}
sxy&=\pi(s^\prime x^\prime y^\prime)&(\text{since}\;\pi\;\text{is a morphism of monoids})\\
&=\pi(x^\prime y^\prime t^\prime)&(\text{since}\; s^\prime x^\prime y^\prime=x^\prime y^\prime t^\prime)\\
&=xyt,&
\end{align*}
\endgroup
and hence $xy\leq_L sxy$. Thus, the pair $(M,S)$ satisfies condition (\ref{condition5}), and hence is of type $\mathcal{R}_1$.
\end{proof}

Figure \ref{arrow-diagram} summarises relations among the pairs $\mathcal{R}_i$, where an edge $\mathcal{R}_i\to\mathcal{R}_j$ means that a pair of type $\mathcal{R}_i$ is of type $\mathcal{R}_j$.
\begin{figure}[H]
\begin{center}
\begin{tikzpicture}[scale=0.6]
\node (r1) at (0,4) {$\mathcal{R}_1$};
\node (r2) at (-2,2) {$\mathcal{R}_2$};
\node (r3) at (-4,0) {$\mathcal{R}_3$};
\node (r4) at (2,2) {$\mathcal{R}_4$};
\node (r5) at (4,0) {$\mathcal{R}_5$};
\node (r6) at (0,0) {$\mathcal{R}_6$};
\node (r7) at (2,-2) {$\mathcal{R}_7$};
\node (r8) at (-2,-2) {$\mathcal{R}_8$};
\node (r9) at (0,-4) {$\mathcal{R}_9$};
\draw[->] (r2)--(r1); \draw[->] (r3)--(r2);
\draw[->] (r6)--(r4); \draw[->] (r8)--(r6);
\draw[->] (r7)--(r5); \draw[->] (r9)--(r7);
\draw[->] (r4)--(r1); \draw[->] (r5)--(r4);
\draw[->] (r6)--(r2); \draw[->] (r7)--(r6);
\draw[->] (r8)--(r3); \draw[->] (r9)--(r8);
\end{tikzpicture}
\caption{A diagram for pairs of types $\mathcal{R}_i$}
\label{arrow-diagram}
\end{center}
\end{figure}
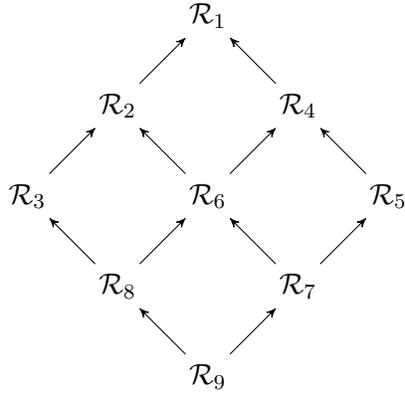

Suppose that $(M,S)$ is a pair that satisfies condition \eqref{condition1}. Then, we have a map $$\alpha:\{(s,t)\in S\times S\mid s\leq_L t\}\to S$$ defined by 
\begin{equation}\label{eqn1}
(s,t)\mapsto (s\backslash t)s.
\end{equation}
Note that $\alpha(s,s)=s$ for all $s\in S$. 

\begin{remark}\label{rmk3}
Let $s,t\in S$ be such that $s\leq_L t$. Then $t=s(s\backslash t)$ and $\alpha(s,t)=(s\backslash t)s$. Suppose that $f$ is a right l.c.m. selector on $S$ in $M$. Since $f(s,t)$ represents the element $s\backslash t$, we have $t =_M sf(s,t)$ and $\alpha(s,t)=_M f(s,t)s$. Thus, $s^{-1}ts =_G f(s,t)s =_G \alpha(s,t)$, where $G$ is the right Gaussian group corresponding to $M$.
\end{remark}

Suppose that $(M,S)$ is a pair that satisfies conditions \eqref{condition2} and \eqref{condition3}. Then, we have a map $\beta:S\times S\to S$ defined by 
\begin{equation}\label{eqn2}
(s,t)\mapsto (s\backslash t)\backslash(s\vee t).
\end{equation}
Note that $\beta(s,s)=s$ for all $s\in S$.

\begin{remark}\label{rmk4}
Let $s$ and $t$ be elements in $S$. Since $(M,S)$ satisfies condition \eqref{condition2}, we have $s\vee t=(s\backslash t)\left((s\backslash t)\backslash(s\vee t)\right)$. In other words, $s(s\backslash t)=(s\backslash t)\left((s\backslash t)\backslash(s\vee t)\right)$. Suppose that $f$ is a right l.c.m. selector on $S$ in $M$. Since $f(s,t)$ represents the element $s\backslash t$ and $\beta(s,t)=(s\backslash t)\backslash(s\vee t)$, we have $sf(s,t) =_M f(s,t)\beta(s,t)$. Thus, $f(s,t)^{-1}sf(s,t) =_G \beta(s,t)$, where $G$ is the right Gaussian group corresponding to $M$. 
\end{remark}

Suppose that $(M,S)$ is a pair that satisfies conditions \eqref{condition1}, \eqref{condition2} and \eqref{condition3}. Let $\alpha$ and $\beta$ be the maps defined by \eqref{eqn1} and \eqref{eqn2}, respectively. Let $G$ be the right Gaussian group corresponding to $M$ and $f$ a right l.c.m. selector on $S$ in $M$. Then, the {\it right $f$-equivalence} on $S^{(S\cup S^{-1})^*}$ is defined to be the equivalence relation generated by the following transformations:\\[8pt]
($\text{R}_\alpha$)\hskip52mm $x^{-1}s^{-1}tsx\longleftrightarrow x^{-1}\alpha(s,t)x$\\[8pt]
where $s,t\in S$ with $s\leq_L t$ and $x\in(S\cup S^{-1})^*$.\\[8pt]
($\text{R}_{f,\beta}$)\hskip43mm $x^{-1}f(s,t)^{-1}sf(s,t)x\longleftrightarrow x^{-1}\beta(s,t)x$\\[8pt]
where $s,t\in S$ and $x\in(S\cup S^{-1})^*$.\\[8pt]
($\text{R}_G$)\hskip60mm $x^{-1}sx\longleftrightarrow y^{-1}sy$\\[8pt]
where $s\in S$ and $x,y\in(S\cup S^{-1})^*$ with $x=_G y$.\\

For $x,y\in S^{(S\cup S^{-1})^*}$, we denote $x\simeq_R^f y$ if $x$ and $y$ are right $f$-equivalent. 

\begin{lemma}\label{lem-f-equivalence}
Let $x,y\in S^{(S\cup S^{-1})^*}$ and $z\in(S\cup S^{-1})^*$. If $x\simeq_R^f y$, then $z^{-1}xz\simeq_R^f z^{-1}yz$.
\end{lemma}

\begin{proof}
Let $x,y\in S^{(S\cup S^{-1})^*}$ and $z\in(S\cup S^{-1})^*$ such that $x\simeq_R^f y$. It is clear that if $x$ and $y$ differ by a single transformation $\text{R}_\alpha$ or $\text{R}_{f,\beta}$ or $\text{R}_G$, then by the transformation itself, $z^{-1}xz\simeq_R^f z^{-1}yz$. Since the right $f$-equivalence on $S^{(S\cup S^{-1})^*}$ is generated by transformations $\text{R}_\alpha$, $\text{R}_{f,\beta}$ and $\text{R}_G$, we have $z^{-1}xz\simeq_R^f z^{-1}yz$ in the general case as well.
\end{proof}

The next result is motivated by \cite[Theorem 1.4]{Kamada1996}.

\begin{theorem}\label{thm-right-equivalence}
Suppose that $(M,S)$ is a pair of type $\mathcal{R}_1$. Let $G$ be the right Gaussian group corresponding to $M$ and $f$ a right l.c.m. selector on $S$ in $M$. Then, two words in $S^{(S\cup S^{-1})^*}$ are right $f$-equivalent if and only if they represent the same element in $G$.
\end{theorem}

We require some lemmas to prove the preceding theorem.

\begin{lemma}\label{lem-1-right-equivalence}
Let $(M,S)$ be a pair that satisfies conditions \eqref{condition1}, \eqref{condition2} and \eqref{condition3}. Let $G$ be the right Gaussian group corresponding to $M$ and $f$ a right l.c.m. selector on $S$ in $M$. Let $s,t\in S$ and $x\in S^*$ be such that $x^{-1}sx =_G t$. Then $x^{-1}sx\simeq_R^f t$.
\end{lemma}

\begin{proof}
Since $M$ is atomic, by \cite[Proposition 2.1]{Dehornoy-Paris-1999}, there exists a map $\nu:M\to \mathbb{Z}_{\ge 0}$ such that $\nu(a)+\nu(b)\leq \nu(ab)$ for all $a,b\in M$, and $\nu(a)=0$ if and only if $a=1$. For $w\in S^*$, let $\bar{w}$ denote the element in $M$ represented by $w$. We prove the lemma by induction on $\nu(\bar{x})$. Suppose that $x^{-1}sx =_G t$ and $\nu(\bar{x})=0$. Then $\bar{x}=1$, and it follows from properties of $\nu$ that $x$ is the empty word. This implies that $s =_M t$, i.e. $s$ and $t$ are the same elements of $S$. Thus $x^{-1}sx\simeq_R^f s\simeq_R^f t$. Suppose that the lemma is true when $\nu(\bar{x})\leq n$. Let $x\in S^*$ be such that $\nu(\bar{x})=n+1$ and $x^{-1}sx =_G t$. Let $x=uy$ for $u\in S$ and $y\in S^*$. Then $x^{-1}sx =_G y^{-1}u^{-1}suy$, and hence $y^{-1}u^{-1}suy =_G t$. Let us consider the following cases:
\begin{enumerate}[(1)]
\item $u\leq_L s$: By transformation $\text{R}_\alpha$, we have $$y^{-1}u^{-1}suy\simeq_R^f y^{-1}\alpha(u,s)y,$$ where $\alpha$ is the map defined by \eqref{eqn1}. Since $x=uy$ and $u\neq 1$, it follows from properties of $\nu$ that $\nu(\bar{y})<\nu(\bar{x})$. In other words, $\nu(\bar{y})\leq n$. Using Remark \ref{rmk3}, we get $$y^{-1}\alpha(u,s)y =_G y^{-1}u^{-1}suy =_G t.$$ By induction hypothesis, we have $$y^{-1}\alpha(u,s)y\simeq_R^f t,$$ and hence $$x^{-1}sx\simeq_R^f y^{-1}u^{-1}suy\simeq_R^f y^{-1}\alpha(u,s)y\simeq_R^f t.$$
\item $u\nleq_L s$: Since $y^{-1}u^{-1}suy =_G t$ and $M$ embeds in $G$, we have $suy =_M uyt$. Thus, the words $suy$ and $uyt$ both represent the common right multiple $\bar{s}\bar{u}\bar{y}=\bar{u}\bar{y}\bar{t}$ of $\bar{s}$ and $\bar{u}$ in $M$. Since $sf(s,u)$ represents the right l.c.m. of $\bar{s}$ and $\bar{u}$, we have $suy =_M sf(s,u)z$ for some $z\in S^*$. By cancellation on the left, we get $uy =_M f(s,u)z$. Thus
\begin{align*}
y^{-1}u^{-1}suy&\simeq_R^f z^{-1}f(s,u)^{-1}sf(s,u)z&(\text{by transformation}\; \text{R}_G)\\
&\simeq_R^f z^{-1}\beta(s,u)z&(\text{by transformation}\; \text{R}_{f,\beta}),
\end{align*} 
where $\beta$ is the map defined by \eqref{eqn2}. Since $u(u\backslash s)=s(s\backslash u)$ and $u\nleq_L s$, we have $s\backslash u\neq1$. In other words, $f(s,u)$ does not represent the identity element of $M$. Since $x =_M uy =_M f(s,u)z$, it follows from properties of $\nu$ that $\nu(\bar{z})<\nu(\bar{x})$. Thus, we have $\nu(\bar{z})\leq n$. Using Remark \ref{rmk4}, we get $$z^{-1}\beta(s,u)z =_G y^{-1}u^{-1}suy =_G t.$$ By induction hypothesis, we have$$z^{-1}\beta(s,u)z\simeq_R^f t.$$ Hence, we obtain $$x^{-1}sx\simeq_R^f y^{-1}u^{-1}suy\simeq_R^f z^{-1}\beta(s,u)z\simeq_R^f t,$$
and the proof is complete. \qedhere
\end{enumerate}
\end{proof}

\begin{lemma}\label{lem-2-right-equivalence}
Let $(M,S)$ be a pair that satisfies conditions (\ref{condition1}), (\ref{condition2}) and (\ref{condition3}). Let $s\in S$ and $x,y\in S^*$ be such that $sx =_M xy$. Then, there exists $t\in S$ such that $y =_M t$.
\end{lemma}

\begin{proof}
By \cite[Proposition 2.1]{Dehornoy-Paris-1999}, we have a map $\nu:M\to \mathbb{Z}_{\ge 0}$ such that $\nu(a)+\nu(b)\leq \nu(ab)$ for all $a,b\in M$, and $\nu(a)=0$ if and only if $a=1$. For $w\in S^*$, let $\bar{w}$ denote the element in $M$ represented by $w$. We prove the lemma by induction on $\nu(\bar{x})$. Suppose $sx =_M xy$ and $\nu(\bar{x})=0$. Then $\bar{x}=1$, and it follows from by the properties of $\nu$ that $x$ is the empty word. This implies that $y =_M s$, and the lemma holds.  Assume that the lemma is true when $\nu(\bar{x})\leq n$. Suppose $s\in S$ and $x,y\in S^*$ be such that $\nu(\bar{x})= n+1$ and $sx =_M xy$. Let $x=uz$ for $u\in S$ and $z\in S^*$. Then $suz =_M uzy$. We now consider two cases below:
\begin{enumerate}[(1)]
\item $u\leq_L s$: Since $x=uz$ and $u\neq 1$, it follows from properties of $\nu$ that $\nu(\bar{z})<\nu(\bar{x})$. In other words, $\nu(\bar{z})\leq n$. Let $\alpha$ be the map defined by \eqref{eqn1}. Then, Remark \ref{rmk3} gives $su =_M u\alpha(u,s)$. Since $suz =_M uzy$, we have $u\alpha(u,s)z =_M uzy$. Cancellation on the left gives $\alpha(u,s)z =_M zy$. By induction hypothesis, there exists $t\in S$ such that $y =_M t$.
\item $u\nleq_L s$: Since $suz =_M uzy$, the words $suz$ and $uzy$ both represent the common right multiple $\bar{s}\bar{u}\bar{z}=\bar{u}\bar{z}\bar{y}$ of $\bar{s}$ and $\bar{u}$ in $M$. Let $f$ be a right l.c.m. selector on $S$ in $M$. Since $sf(s,u)$ represents the right l.c.m. of $\bar{s}$ and $\bar{u}$, we have $suz =_M sf(s,u)w$ for some  $w\in S^*$. Cancellation on the left gives $uz =_M f(s,u)w$. Since $u\nleq_L s$, we have $s\backslash u\neq1$. Thus, $f(s,u)$ does not represent the identity element of $M$. Since $x =_M uz =_M f(s,u)w$, it follows from properties of $\nu$ that $\nu(\bar{w})<\nu(\bar{x})$, and hence $\nu(\bar{w})\leq n$. Let $\beta$ be the map defined by \eqref{eqn2}. Then
\begin{align*}
f(s,u)\beta(s,u)w& =_M sf(s,u)w&(\text{by Remark \ref{rmk4}})\\
&=_M suz&(\text{since}\; f(s,u)w =_M uz)\\
&=_M uzy&(\text{since}\; suz =_M uzy)\\
&=_M f(s,u)wy&(\text{since}\; uz =_M f(s,u)w).
\end{align*}
Left cancellation gives $\beta(s,u)w =_M wy$. Thus, by induction hypothesis, there exists $t\in S$ such that $y =_M t$, and the proof is complete. \qedhere
\end{enumerate}
\end{proof}

\begin{lemma}\label{lem-3-right-equivalence}
Let $(M,S)$ be a pair of type $\mathcal{R}_1$. Then, for each $s\in S$ and each $x\in S^*$, there exist $t=t(s,x)$ in $S$ and $y=y(s,x)$ in $S^*$ such that $sxy =_M xyt$.
\end{lemma}

\begin{proof}
Let $s\in S$ and $x\in S^*$. For $w\in S^*$, let $\bar{w}$ denote the element in $M$ represented by $w$. By condition \eqref{condition5}, there exist $y=y(s,x)$ and $z=z(s,x)$ in $S^*$ such that $\bar{s}\bar{x}\bar{y}=\bar{x}\bar{y}\bar{z}$. Thus, we have $sxy =_M xyz$. By Lemma \ref{lem-2-right-equivalence}, there exists $t\in S$ such that $z =_M t$. Hence, we have $sxy =_M xyt$.
\end{proof}

\begin{lemma}\label{lem-4-right-equivalence}
Let $(M,S)$ be a pair of right Gaussian monoid $M$ and a finite generating set $S$ for $M$, and $G$ the group of fractions of $M$. Then, for each element $x\in(S\cup S^{-1})^*$, there exist $y,z\in S^*$ such that $x =_G yz^{-1}$.
\end{lemma}

\begin{proof}
Let $f$ be a right l.c.m. selector on $S$ in $M$. Then, by \cite[Theorem 4.2]{Dehornoy-Paris-1999}, $(S,f)$ satisfies condition $\text{III}_{\text{R}}$, and hence $R_R^f(x)$ exists for any $x\in(S\cup S^{-1})^*$. In other words, for every $x\in(S\cup S^{-1})^*$, there exist $y,z\in S^*$ such that $x =_G yz^{-1}$. See the definitions just after Proposition 3.2 in \cite{Dehornoy-Paris-1999} for definitions of $R_R^f(x)$ and condition $\text{III}_{\text{R}}$.
\end{proof}

\begin{lemma}\label{lem-5-right-equivalence}
Let $(M,S)$ be a pair of type $\mathcal{R}_1$, $G$ the right Gaussian group corresponding to $M$ and $f$ a right l.c.m. selector on $S$ in $M$. Let $s,t\in S$ and $x\in(S\cup S^{-1})^*$ be such that $x^{-1}sx =_G t$. Then $x^{-1}sx\simeq_R^f t$.
\end{lemma}
\begin{proof}
Let $s,t\in S$ and $x\in(S\cup S^{-1})^*$ be such that $x^{-1}sx =_G t$. By Lemma \ref{lem-4-right-equivalence}, there exist $y,z\in S^*$ such that $x =_G yz^{-1}$. Further, by Lemma \ref{lem-3-right-equivalence}, there exist $u\in S$ and $w\in S^*$ such that $syw =_M ywu$. Thus 
$w^{-1}y^{-1}syw =_G u$, and hence by Lemma \ref{lem-1-right-equivalence}, $$w^{-1}y^{-1}syw\simeq_R^f u.$$ Since $x^{-1}sx =_G t$ and $x =_G yz^{-1}$, we have $y^{-1}sy =_G z^{-1}tz$. Thus $w^{-1}z^{-1}tzw =_Gw^{-1}y^{-1}syw =_G u$. Again, by Lemma \ref{lem-1-right-equivalence}, we have $$w^{-1}z^{-1}tzw\simeq_R^f u,$$ and consequently $$w^{-1}y^{-1}syw\simeq_R^f w^{-1}z^{-1}tzw.$$ Now, using transformation $\text{R}_G$ and Lemma \ref{lem-f-equivalence} alternatively, we get
\begin{align*}
x^{-1}sx&\simeq_R^f zww^{-1}y^{-1}syww^{-1}z^{-1}&(\text{since}\; x =_G yww^{-1}z^{-1})\\
&\simeq_R^f zww^{-1}z^{-1}tzww^{-1}z^{-1}&(\text{since}\; w^{-1}y^{-1}syw\simeq_R^f w^{-1}z^{-1}tzw)\\
&\simeq_R^f t&(\text{since}\; zww^{-1}z^{-1} =_G \epsilon,\; \text{where}\; \epsilon\; \text{is the empty word}).
\end{align*}
This proves the lemma.
\end{proof}
\medskip

\begin{proof}[Proof of Theorem \ref{thm-right-equivalence}]
It follows from remarks \ref{rmk3} and \ref{rmk4} that if two words in $S^{(S\cup S^{-1})^*}$ are right $f$-equivalent, then they represent the same element in $G$. For the converse part, let $x,y\in(S\cup S^{-1})^*$ and $s,t\in S$ be such that $x^{-1}sx =_G y^{-1}ty$. Then $yx^{-1}sxy^{-1} =_G t$. By Lemma \ref{lem-5-right-equivalence}, we have $yx^{-1}sxy^{-1}\simeq_R^f t$. Thus, by Lemma \ref{lem-f-equivalence} together with transformation $\text{R}_G$, we get $y^{-1}ty\simeq_R^f y^{-1}yx^{-1}sxy^{-1}y\simeq_R^f x^{-1}sx$.
\end{proof}

The following main result is inspired by \cite[Theorem 3.1]{kamadamatsumoto} for the quandle of cords of the plane.

\begin{theorem}\label{thm-present-dehn-quandle-right-gaussian-group}
Suppose that $(M,S)$ is a pair of type $\mathcal{R}_1$. Let $\alpha$ and $\beta$ be maps defined by \eqref{eqn1} and \eqref{eqn2}, respectively. Let $G$ be the right Gaussian group corresponding to $M$ and $f$ a right l.c.m. selector on $S$ in $M$. For $(s,t)\in S\times S$ with $s\neq t$, let $f_1(s,t),f_2(s,t),\ldots,f_{n_{st}}(s,t)$ be elements in $S$ such that $f(s,t)=f_1(s,t)f_2(s,t)\cdots f_{n_{st}}(s,t)$. Then, the Dehn quandle $\mathcal{D}(S^G)$ has a presentation with $S$ as the set of generators and defining relations as follows:
\begin{enumerate}[(1)]
\item $\alpha(s,t)*s=t$\quad if $s\leq_L t$,\label{relation1}
\item $\beta(s,t)*f_{n_{st}}(s,t)*f_{n_{st}-1}(s,t)*\cdots *f_1(s,t)=s$, \label{relation2}
\item $u*f_{n_{st}}(s,t)*f_{n_{st}-1}(s,t)*\cdots *f_1(s,t)*s=u*f_{n_{ts}}(t,s)*f_{n_{ts}-1}(t,s)*\cdots *f_1(t,s)*t$\label{relation3}
\end{enumerate}
for $s,t,u\in S$ with $s\neq t$.
\end{theorem}

Proving Theorem \ref{thm-present-dehn-quandle-right-gaussian-group} requires some preparation in the form of following lemmas.

\begin{lemma}\label{lem-1-present-dehn-quandle-right-gaussian-group}
Let $(M,S)$ be a pair that satisfies conditions \eqref{condition1}, \eqref{condition2} and \eqref{condition3}. Let $\alpha$ and $\beta$ be maps defined by \eqref{eqn1} and \eqref{eqn2}, respectively. Let $f$ be a right l.c.m. selector on $S$ in $M$. Then, the {\it right $f$-equivalence} on $S^{(S\cup S^{-1})^*}$ is same as the equivalence relation generated by the following transformations:\\[8pt]
($\text{T}_\alpha$)\hskip52mm $xs\alpha(s,t)s^{-1}x^{-1}\longleftrightarrow xtx^{-1},$\\[8pt]
where $s,t\in S$ with $s\leq_L t$ and $x\in(S\cup S^{-1})^*$.\\[8pt]
($\text{T}_{f,\beta}$)\hskip43mm $xf(s,t)\beta(s,t)f(s,t)^{-1}x^{-1}\longleftrightarrow xsx^{-1},$\\[8pt]
where $s,t\in S$ and $x\in(S\cup S^{-1})^*$.\\[8pt]
($\text{T}_0$)\hskip60mm $xsx^{-1}\longleftrightarrow ysy^{-1},$\\[8pt]
where $s\in S$ and $x,y\in (S\cup S^{-1})^*$ are words differing by a single use of a relation of the form $tt^{-1}=\epsilon$ or $t^{-1}t=\epsilon$ for $t\in S$. Here $\epsilon$ denotes the empty word.\\[8pt]
($\text{T}_f$)\hskip60mm $xsx^{-1}\longleftrightarrow ysy^{-1},$\\[8pt]
where $s\in S$ and $x,y\in (S\cup S^{-1})^*$ are words differing by a single use of a relation of the form $uf(u,v)=vf(v,u)$ for $u,v\in S$.
\end{lemma}

\begin{proof}
For $x,y\in S^{(S\cup S^{-1})^*}$, we write $x\simeq_T^f y$ if $x$ and $y$ are equivalent under the equivalence relation generated by $\text{T}_\alpha$, $\text{T}_{f,\beta}$, $\text{T}_0$ and $\text{T}_f$. By \cite[Theorem 4.1]{Dehornoy-Paris-1999}, $M$ has a complemented presentation $$\left\langle S\mid sf(s,t)=tf(t,s)\;\text{for}\;s,t\in S\right\rangle.$$ Let $G$ be the right Gaussian group corresponding to $M$. Then, $G$ is the quotient of the free monoid $(S\cup S^{-1})^*$ by relations $ss^{-1}=1$, $s^{-1}s=1$ and $sf(s,t)=tf(t,s)$ for $s,t\in S$. To prove the lemma, we need to show that each single transformation $\text{R}_\alpha$, $\text{R}_{f,\beta}$ and $\text{R}_G$ can be obtained from a combination of transformations $\text{T}_\alpha$, $\text{T}_{f,\beta}$, $\text{T}_0$ and $\text{T}_f$, and vice-versa.
\par
Let $s\in S$ and $x,y\in(S\cup S^{-1})^*$. If $x=_G y$, then there is a sequence $x=x_1,x_2,\ldots, x_n=y$ of words such that $x_i$ and $x_{i+1}$ differ by a single use of a relation of the form $uu^{-1}=\epsilon$ or $u^{-1}u=\epsilon$ or $uf(u,v)=vf(v,u)$ for $u,v\in S$. In this case, by virtue of transformations $\text{T}_0$ and $\text{T}_f$, we have $x_isx_i^{-1}\simeq_T^f x_{i+1}sx_{i+1}^{-1}$ for all $i$. Consequently,  $xsx^{-1}\simeq_T^f ysy^{-1}$. This suggests that, for the equivalence relation $\simeq_T^f$, transformations $\text{T}_0$ and $\text{T}_f$ can be replaced by the following transformation:\\[8pt]
($\text{T}_G$)\hskip60mm $xsx^{-1}\longleftrightarrow ysy^{-1},$\\[8pt]
where $s\in S$ and $x,y\in(S\cup S^{-1})^*$ with $x=_G y$. Note that if $w\in(S\cup S^{-1})^*$, then $w^{-1}\in(S\cup S^{-1})^*$ and $(w^{-1})^{-1}=w$. Further, $x^{-1}=_G y^{-1}$ if and only if $x =_G y$. Thus, transformations $\text{T}_G$ and $\text{R}_G$ are the same. It is now sufficient to prove the following statements:
\begin{enumerate}[(1)]
\item If $x,y\in S^{(S\cup S^{-1})^*}$ are words differing by a single use of $\text{R}_\alpha$ or $\text{R}_{f,\beta}$, then $x\simeq_T^f y$.\label{statement1}
\item If $x,y\in S^{(S\cup S^{-1})^*}$ are words differing by a single use of $\text{T}_\alpha$ or $\text{T}_{f,\beta}$, then $x\simeq_R^f y$.\label{statement2}
\end{enumerate}

With analogy to Lemma \ref{lem-f-equivalence}, we can say that if $y, z \in S^{(S\cup S^{-1})^*}$ and $x\in(S\cup S^{-1})^*$ such that $y\simeq_T^f z$, then $xyx^{-1}\simeq_T^f xzx^{-1}$. Since $x^{-1}\in(S\cup S^{-1})^*$ and $(x^{-1})^{-1}=x$, we also have $x^{-1}yx\simeq_T^f x^{-1}zx$ whenever $y\simeq_T^f z$. This together with the following statements will prove statement \eqref{statement1}.
\begin{align}
&s^{-1}ts\simeq_T^f \alpha(s,t)&\text{for}\; s,t\in S\; \text{with}\; s\leq_L t.\label{eqn7}\\
&f(s,t)^{-1}sf(s,t)\simeq_T^f \beta(s,t)&\text{for}\; s,t\in S.\label{eqn8}
\end{align}
Thus, to prove statement \eqref{statement1}, it is enough to prove \eqref{eqn7} and \eqref{eqn8}. Using transformations $\text{T}_\alpha$ and $\text{T}_{f,\beta}$, we have the following:
\begin{align}
&s\alpha(s,t)s^{-1}\simeq_T^f t&\text{for}\; s,t\in S\; \text{with}\; s\leq_L t.\label{eqn16}\\
&f(s,t)\beta(s,t)f(s,t)^{-1}\simeq_T^f s&\text{for}\; s,t\in S.\label{eqn17}
\end{align}
Let $x\in(S\cup S^{-1})^*$ and $y,z\in S^{(S\cup S^{-1})^*}$ be such that $xyx^{-1}\simeq_T^f z$. Then
\begin{align*}
x^{-1}zx &\simeq_T^f x^{-1}xyx^{-1}x&(\text{since}\; z\simeq_T^f xyx^{-1})\\
&\simeq_T^f y&(\text{by transformation}\; \text{T}_G).
\end{align*}
Thus, $x^{-1}zx \simeq_T^f y$ whenever $xyx^{-1}\simeq_T^f z$. This together with \eqref{eqn16} and \eqref{eqn17} implies \eqref{eqn7} and \eqref{eqn8}, respectively. This proves the statement \eqref{statement1}. Statement \eqref{statement2} can be proved similarly.
\end{proof}

The following lemma is a direct consequence of quandle axioms. See, for example, \cite[Lemma 4.4.8]{Winker1984}.

\begin{lemma}\label{lem-2-present-dehn-quandle-right-gaussian-group}
Let $Q$ be quandle. Then, any left associated product $$u_0*^{\delta_1}u_1*^{\delta_2}\cdots*^{\delta_m}u_m *^{\eta_1}v_1*^{\eta_2}v_2*^{\eta_3}\cdots*^{\eta_n}v_n$$ can be written as 
$$(u_0*^{\eta_1}v_1*^{\eta_2}v_2*^{\eta_3}\cdots*^{\eta_n}v_n)*^{\delta_1}(u_1*^{\eta_1}v_1*^{\eta_2}v_2*^{\eta_3}\cdots*^{\eta_n}v_n)*^{\delta_2}\cdots*^{\delta_m}(u_m*^{\eta_1}v_1*^{\eta_2}v_2*^{\eta_3}\cdots*^{\eta_n}v_n)$$
where $u_i,v_j\in Q$ and $\delta_i, \eta_j \in\{-1,1\}$.
\end{lemma}
\medskip

\begin{proof}[Proof of Theorem \ref{thm-present-dehn-quandle-right-gaussian-group}]
Since $S$ generates $G$, it follows from \cite[Proposition 3.2]{Dhanwani-Raundal-Singh-2021} that $S$ also generates the Dehn quandle $\mathcal{D}(S^G)$. Note that relations \eqref{relation1}, \eqref{relation2} and \eqref{relation3} (in the statement of the theorem) can be written respectively as follows:
\begin{eqnarray}
s\alpha(s,t)s^{-1}&=& t\quad \text{if}\; s\leq_L t,\label{eqn21}\\
f(s,t)\beta(s,t)f(s,t)^{-1}&=& s, \label{eqn22}\\
sf(s,t)uf(s,t)^{-1}s^{-1}&=& t f(t,s)uf(t,s)^{-1}t^{-1}\label{eqn23}
\end{eqnarray}
for $s,t,u\in S$ with $s\neq t$. By remarks \ref{rmk3} and \ref{rmk4}, relations \eqref{eqn21} and \eqref{eqn22} are in fact relations in $\mathcal{D}(S^G)$, and hence so are relations \eqref{relation1} and \eqref{relation2}. Also, since $sf(s,t) =_G tf(t,s)$ for any $(s,t)\in S\times S$, \eqref{eqn23} is a relation in $\mathcal{D}(S^G)$, and hence so is a relation as in (\ref{relation3}). By Theorem \ref{thm-right-equivalence} together with Lemma \ref{lem-1-present-dehn-quandle-right-gaussian-group}, it is sufficient to prove that relations in $\mathcal{D}(S^G)$ corresponding to transformations $\text{T}_\alpha$, $\text{T}_{f,\beta}$, $\text{T}_0$ and $\text{T}_f$ can be obtained by relations (\ref{relation1}), (\ref{relation2}) and (\ref{relation3}). Note that a relation corresponding to $\text{T}_0$ is a trivial relation. Also, a relation corresponding to $\text{T}_\alpha$ or $\text{T}_{f,\beta}$ is trivial if $s=t$, and a relation corresponding to $\text{T}_f$ is trivial if $u=v$. The remaining relations can be written as follows:
{\small
\begin{align}
&\alpha(s,t)*s*^{\delta_1}u_1*^{\delta_2}u_2*^{\delta_3}\cdots*^{\delta_m}u_m=t*^{\delta_1}u_1*^{\delta_2}u_2*^{\delta_3}\cdots*^{\delta_m}u_m\quad \text{if}\; s\leq_L t,\label{eqn18}\\
&\beta(s,t)*f_{n_{st}}(s,t)*f_{n_{st}-1}(s,t)*\cdots *f_1(s,t)*^{\delta_1}u_1*^{\delta_2}u_2*^{\delta_3}\cdots*^{\delta_m}u_m\label{eqn19}\\
&=s*^{\delta_1}u_1*^{\delta_2}u_2*^{\delta_3}\cdots*^{\delta_m}u_m,\quad \nonumber\\
&u_0*^{\delta_1}u_1*^{\delta_2}\cdots*^{\delta_m}u_m*f_{n_{st}}(s,t)*f_{n_{st}-1}(s,t)*\cdots *f_1(s,t)*s*^{\epsilon_1}v_1*^{\epsilon_2}v_2*^{\epsilon_3}\cdots*^{\epsilon_n}v_n\label{eqn20}\\
&=u_0*^{\delta_1}u_1*^{\delta_2}\cdots*^{\delta_m}u_m*f_{n_{ts}}(t,s)*f_{n_{ts}-1}(t,s)*\cdots *f_1(t,s)*t*^{\epsilon_1}v_1*^{\epsilon_2}v_2*^{\epsilon_3}\cdots*^{\epsilon_n}v_n\nonumber
\end{align}}
for $s,t,u_i,v_i\in S$ with $s\neq t$ and $\delta_i,\epsilon_i\in\{-1,1\}$. Here, note that \eqref{eqn18}, \eqref{eqn19} and \eqref{eqn20} are non-trivial relations corresponding to  transformations $\text{T}_\alpha$, $\text{T}_{f,\beta}$ and $\text{T}_f$, respectively. Relations \eqref{eqn18} and \eqref{eqn19} can be obtained by relations (\ref{relation1}) and (\ref{relation2}), respectively. By right cancellation, \eqref{eqn20} takes the form
\begin{eqnarray}
&&u_0*^{\delta_1}u_1*^{\delta_2}\cdots*^{\delta_m}u_m*f_{n_{st}}(s,t)*f_{n_{st}-1}(s,t)*\cdots *f_1(s,t)*s\label{eqn24}\\
&=&u_0*^{\delta_1}u_1*^{\delta_2}\cdots*^{\delta_m}u_m*f_{n_{ts}}(t,s)*f_{n_{ts}-1}(t,s)*\cdots *f_1(t,s)*t.\nonumber
\end{eqnarray}
It follows from Lemma \ref{lem-2-present-dehn-quandle-right-gaussian-group} that \eqref{eqn24} can be written as
\begin{equation}
v_0*^{\delta_1}v_1*^{\delta_2}\cdots*^{\delta_m}v_m=w_0*^{\delta_1}w_1*^{\delta_2}\cdots*^{\delta_m}w_m,\label{eqn25}
\end{equation}
where $$v_i=u_i*f_{n_{st}}(s,t)*f_{n_{st}-1}(s,t)*\cdots *f_1(s,t)*s$$
and
$$w_i=u_i*f_{n_{ts}}(t,s)*f_{n_{ts}-1}(t,s)*\cdots *f_1(t,s)*t.$$ It follows from relations as in \eqref{relation3} that $v_i=w_i$ for each $i$. Thus \eqref{eqn25} can be obtained by relations as in (\ref{relation3}), and hence so are \eqref{eqn24} and \eqref{eqn20}. This completes the proof of the theorem.
\end{proof}

By Proposition \ref{prop-pairs-of-types-r}, theorems \ref{thm-right-equivalence} and \ref{thm-present-dehn-quandle-right-gaussian-group} are also true for pairs of types $\mathcal{R}_2$ through $\mathcal{R}_9$. For $i\in\{4,5,6,7\}$, let $(M,A)$ be a pair of type $\mathcal{R}_i$. Then, for $a,b\in A$, $a\leq_L b$ if and only if $a=b$ (see Remark \ref{rmk2}). Thus, we have the following corollary of Theorem \ref{thm-present-dehn-quandle-right-gaussian-group}.

\begin{corollary}\label{cor-present-right-gaussian-quandle}
For $i\in\{4,5,6,7\}$, let $M$ be a right Gaussian monoid of type $\mathcal{R}_i$. Let $A$ be the set of atoms in $M$ and $\beta:A\times A\to A$ the map defined by $(a,b)\mapsto (a\backslash b)\backslash(a\vee b)$. Let $G$ be the right Gaussian group corresponding to $M$ and $f$ a right l.c.m. selector on $A$ in $M$. For $(a,b)\in A\times A$ with $a\neq b$, let $f_1(a,b),f_2(a,b),\ldots,f_{n_{ab}}(a,b)$ be elements in $A$ such that $f(a,b)=f_1(a,b)f_2(a,b)\cdots f_{n_{ab}}(a,b)$. Then, the right Gaussian quandle $\mathcal{D}(A^G)$ has a presentation with $A$ as its set of generators and defining relations as follows:
\begin{enumerate}[(1)]
\item $\beta(a,b)*f_{n_{ab}}(a,b)*f_{n_{ab}-1}(a,b)*\cdots *f_1(a,b)=a$,
\item $c*f_{n_{ab}}(a,b)*f_{n_{ab}-1}(a,b)*\cdots *f_1(a,b)*a=c*f_{n_{ba}}(b,a)*f_{n_{ba}-1}(b,a)*\cdots *f_1(b,a)*b$
\end{enumerate}
for $a,b,c\in A$ with $a\neq b$.
\end{corollary}

For $i\in\{8,9\}$, let $(N,B)$ be a pair of type $\mathcal{R}_i$. Then, for $a,b\in B$, $a\leq_L b$ if and only if $a=b$. Thus, we have the following corollary.

\begin{corollary}\label{cor-present-garside-quandle-r}
For $i\in\{8,9\}$, let $N$ be a Garside monoid of type $\mathcal{R}_i$. Let $B$ be the set of atoms in $N$ and $\beta:B\times B\to B$ the map defined by $(a,b)\mapsto (a\backslash b)\backslash(a\vee b)$. Let $H$ be the Garside group corresponding to $N$ and $f$ a right l.c.m. selector on $B$ in $N$. For $(a,b)\in B\times B$ with $a\neq b$, let $f_1(a,b),f_2(a,b),\ldots,f_{n_{ab}}(a,b)$ be elements in $B$ such that $f(a,b)=f_1(a,b)f_2(a,b)\cdots f_{n_{ab}}(a,b)$. Then, the Garside quandle $\mathcal{D}(B^H)$ has a presentation with $B$ as its set of generators and defining relations as:
\begin{enumerate}[(1)]
\item $\beta(a,b)*f_{n_{ab}}(a,b)*f_{n_{ab}-1}(a,b)*\cdots *f_1(a,b)=a$,
\item $c*f_{n_{ab}}(a,b)*f_{n_{ab}-1}(a,b)*\cdots *f_1(a,b)*a=c*f_{n_{ba}}(b,a)*f_{n_{ba}-1}(b,a)*\cdots *f_1(b,a)*b$
\end{enumerate}
for $a,b,c\in A$ with $a\neq b$.
\end{corollary}
\medskip

\subsection{Presentations of left Gaussian quandles and Garside quandles of certain types}

This subsection is a left analogue of the preceding section and we present it without details. We set the following notations:
\begin{itemize}
\item $M$ - a left Gaussian monoid;
\item $S$ - a finite generating set for $M$;
\item $A$ - the set of atoms in $M$;
\item $N$ - a Garside monoid;
\item $T$ - a finite generating set for $N$;
\item $B$ - the set of atoms in $N$;
\item $\Delta$ - a Garside element in $N$;
\item $(M,S)$, $(M,A)$, $(N,T)$, $(N,B)$, $(N,\Delta)$ and $(N,T,\Delta)$ - pairs and a triple of objects with the meaning above.
\end{itemize}

For a pair $(M,S)$, we assume throughout this subsection that elements in $S$ are pairwise distinct in $M$. The same is assumed in case of a pair $(N,T)$. For elements $x$ and $y$ in a left Gaussian monoid, denote the right g.c.d. and the left l.c.m. of $x$ and $y$ respectively by $x\widetilde{\wedge}y$ and $x\widetilde{\vee}y$. The {\it left residue} of $x$ in $y$ (denoted by $y/x$) is the unique element $z$ satisfying $x\widetilde{\vee}y=zx$. Thus, we have
\begin{equation*}
x\widetilde{\vee}y=(y/x)x=(x/y)y\,.
\end{equation*}

Consider the following conditions on a pair $(M,S)$:
\begin{enumerate}[(i)]\setcounter{enumi}{5}
\item $t(s/t)\in S$ whenever $(s,t)\in S\times S$ and $t\leq_R s$.\label{condition6}
\item $s/t\leq_R s\widetilde{\vee} t$ for all $(s,t)\in S\times S$.\label{condition7}
\item $(s\widetilde{\vee} t)/(s/t)\in S$ for all $(s,t)\in S\times S$.\label{condition8}
\item $M$ has a finite homogeneous presentation $\left\langle S\mid R\right\rangle$.\label{condition9}
\item For each $s\in S$ and each $x\in M$, there exists $y=y(s,x)$ in $M$ such that $yx\leq_R yxs$.\label{condition10}
\end{enumerate}

Define the following terms:

\begin{itemize}
\item A pair $(M,S)$ is of
\begin{itemize}
\item {\it type $\mathcal{L}_1$} if it satisfies conditions (\ref{condition6}), (\ref{condition7}), (\ref{condition8}) and (\ref{condition10}).
\item {\it type $\mathcal{L}_2$} if it satisfies conditions (\ref{condition6}), (\ref{condition7}) and (\ref{condition8}), and if there exist a triple $(N,T,\Delta)$ with $T\Delta=\Delta T$ and an epimorphism $\pi:(N,T)\twoheadrightarrow (M,S)$.
\end{itemize}

\item A pair $(N,T)$ is of {\it type $\mathcal{L}_3$} if it satisfies conditions (\ref{condition6}), (\ref{condition7}) and (\ref{condition8}), and there exists a Garside element $\Delta\in N$ such that $T\Delta=\Delta T$.

\item A pair $(M,A)$ is of
\begin{itemize}
\item {\it type $\mathcal{L}_4$} if it satisfies conditions (\ref{condition7}), (\ref{condition8}) and (\ref{condition10}).
\item {\it type $\mathcal{L}_5$} if it satisfies conditions (\ref{condition7}), (\ref{condition9}) and (\ref{condition10}).
\item {\it type $\mathcal{L}_6$} if it satisfies conditions (\ref{condition7}) and (\ref{condition8}), and if there exist a triple $(N,T,\Delta)$ with $T\Delta=\Delta T$ and an epimorphism $\pi:(N,T)\twoheadrightarrow (M,A)$.
\item {\it type $\mathcal{L}_7$} if it satisfies conditions (\ref{condition7}) and (\ref{condition9}), and if there exist a triple $(N,T,\Delta)$ with $T\Delta=\Delta T$ and an epimorphism $\pi:(N,T)\twoheadrightarrow (M,A)$.
\end{itemize}

\item A pair $(N,B)$ is of
\begin{itemize}
\item {\it type $\mathcal{L}_8$} if it satisfies conditions (\ref{condition7}) and (\ref{condition8}).
\item {\it type $\mathcal{L}_9$} if it satisfies conditions (\ref{condition7}) and (\ref{condition9}).
\end{itemize}
\end{itemize}

For $i=4,5,6,7$, we say a left Gaussian monoid $M$ is of {\it type $\mathcal{L}_i$} if the pair $(M,A)$ is of type $\mathcal{L}_i$. In this case, we also say that the left Gaussian group $G$ corresponding to $M$ and the left Gaussian quandle $\mathcal{D}(A^G)$ are of {\it type $\mathcal{L}_i$}. Similarly, for $i=8,9$, we say a Garside monoid $N$ is of type {\it type $\mathcal{L}_i$} if the pair $(N,B)$ is of type $\mathcal{L}_i$. In this case, we also say that the Garside group $H$ corresponding to $N$ and the Garside quandle $\mathcal{D}(B^H)$ are of {\it type $\mathcal{L}_i$}.

The next proposition is analogous to Proposition \ref{prop-pairs-of-types-r}.

\begin{proposition}\label{prop-pairs-of-types-l}
Pairs of types $\mathcal{L}_2$ through $\mathcal{L}_9$ are of type $\mathcal{L}_1$.
\end{proposition}

Let $(M,S)$ be a pair that satisfies condition \eqref{condition6}. Then, we have a map $$\alpha:\{(s,t)\in S\times S\mid t\leq_R s\}\to S$$ defined by
\begin{equation}\label{eqn3}
(s,t)\mapsto t(s/t).
\end{equation}
Note that $\alpha(s,s)=s$ for all $s\in S$.

\begin{remark}
Let $s,t\in S$ be such that $t\leq_R s$. Then $s=(s/t)t$ and $\alpha(s,t)=t(s/t)$. Suppose $f$ be a left l.c.m. selector on $S$ in $M$. Since $f(s,t)$ represents the element $s/t$, we have $s=_M f(s,t)t$ and $\alpha(s,t)=_M tf(s,t)$. Thus $tst^{-1}=_G tf(s,t)=_G \alpha(s,t)$, where $G$ is the left Gaussian group corresponding to $M$.
\end{remark}

Let $(M,S)$ be a pair that satisfies conditions \eqref{condition7} and \eqref{condition8}. Then, we have a map $\beta:S\times S\to S$ defined by
\begin{equation}\label{eqn4}
(s,t)\mapsto (s\widetilde{\vee} t)/(s/t).
\end{equation}
Note that $\beta(s,s)=s$ for all $s\in S$.

\begin{remark}
Let $s$ and $t$ be elements in $S$. Then $s\widetilde{\vee} t=\left((s\widetilde{\vee} t)/(s/t)\right)(s/t)$. In other words, $(s/t)t=\left((s\widetilde{\vee} t)/(s/t)\right)(s/t)$. Suppose $f$ be a left l.c.m. selector on $S$ in $M$. Since $f(s,t)$ represents the element $s/t$ and $\beta(s,t)=(s\widetilde{\vee} t)/(s/t)$, we have $f(s,t)t =_M \beta(s,t)f(s,t)$. Thus $f(s,t)tf(s,t)^{-1} =_G \beta(s,t)$, where $G$ is the left Gaussian group corresponding to $M$.
\end{remark}

The next theorem is analogous to Theorem \ref{thm-present-dehn-quandle-right-gaussian-group}.

\begin{theorem}\label{thm-present-dehn-quandle-left-gaussian-group}
Suppose $(M,S)$ be a pair of type $\mathcal{L}_1$. Let $\alpha$ and $\beta$ be maps defined by \eqref{eqn3} and \eqref{eqn4}, respectively. Suppose $G$ be the left Gaussian group corresponding to $M$ and $f$ a left l.c.m. selector on $S$ in $M$. For $(s,t)\in S\times S$ with $s\neq t$, let $f_1(s,t),f_2(s,t),\ldots,f_{n_{st}}(s,t)$ be elements in $S$ such that $f(s,t)=f_1(s,t)f_2(s,t)\cdots f_{n_{st}}(s,t)$. Then, the Dehn quandle $\mathcal{D}(S^G)$ has a presentation with $S$ as its set of generators and defining relations as:
\begin{enumerate}[(1)]
\item $s*t=\alpha(s,t)$\quad if $t\leq_R s$,
\item $t*f_{n_{st}}(s,t)*f_{n_{st}-1}(s,t)*\cdots *f_1(s,t)=\beta(s,t)$,
\item $u*t*f_{n_{st}}(s,t)*f_{n_{st}-1}(s,t)*\cdots *f_1(s,t)=u*s*f_{n_{ts}}(t,s)*f_{n_{ts}-1}(t,s)*\cdots *f_1(t,s)$
\end{enumerate}
for $s,t,u\in S$ with $s\neq t$.
\end{theorem}

In view of Proposition \ref{prop-pairs-of-types-l}, Theorem \ref{thm-present-dehn-quandle-left-gaussian-group} is also true for pairs of types $\mathcal{L}_2$ through $\mathcal{L}_9$. Further, we can deduce corollaries of Theorem \ref{thm-present-dehn-quandle-left-gaussian-group} analogous to corollaries \ref{cor-present-right-gaussian-quandle} and \ref{cor-present-garside-quandle-r}.
\medskip

\subsection{Examples of presentations of Garside quandles}\label{examples of presentations}
Theorem \ref{thm-present-dehn-quandle-right-gaussian-group} and Corollary \ref{cor-present-garside-quandle-r} can be used to write presentations of certain Garside quandles.

\begin{example}\label{spherical artin}
Let $\mathcal{A}$ be an Artin group with Artin presentation 
\begin{equation}\label{presentation of artin group}
\left\langle S~|~ (st)_{m_{st}}=(ts)_{m_{ts}}\; \text{for}\; s,t\in S\; \text{with}\; s\neq t\right\rangle,
\end{equation}
 where $S$ is a finite set, $(uv)_k$ denotes the word $uvuvu\cdots$ of length $k$, $m_{st}=m_{ts}\geq 2$ for all $s\neq t$ and $m_{st}=\infty$ if there is no relation. Let $\mathcal{M}$ be the monoid with the same presentation as above, i.e. $M$ is the quotient of the free monoid $S^*$ by relations $(st)_{m_{st}}=(ts)_{m_{ts}}$ for all $s\neq t$. We refer to such a monoid as an {\it Artin monoid} and call it of {\it spherical type} if the corresponding Artin group is of spherical type. Further, we call the Dehn quandle $\mathcal{D}(S^{\mathcal{A}})$ as an {\it Artin quandle} and say that it is of {\it spherical type} if the Artin group $\mathcal{A}$ is of spherical type. Note that, in these definitions, $S$ is an Artin generating set for $\mathcal{A}$. 
\par

For the further discussion, we assume that $\mathcal{M}$ is a spherical Artin monoid. It follows from \cite{Brieskorn-Saito-1972} (see also \cite[Example 1]{Dehornoy-Paris-1999}) that $\mathcal{M}$ is a Garside monoid. Note that, for $s\neq t$, the words $(st)_{m_{st}}$ and $(ts)_{m_{ts}}$ both represent the right l.c.m. of $s$ and $t$. It is easy to see that the pair $(\mathcal{M},S)$ satisfies conditions \eqref{condition2} and \eqref{condition4}. By Lemma \ref{lem-conditions-ii-and-iv}, the set $S$ is the set of atoms in $\mathcal{M}$. Thus, $\mathcal{M}$ is a Garside monoid of type $\mathcal{R}_9$, and hence the Artin quandle $\mathcal{D}(S^{\mathcal{A}})$ is a Garside quandle of type $\mathcal{R}_9$. Again by Lemma \ref{lem-conditions-ii-and-iv}, the pair $(\mathcal{M},S)$ satisfies condition (\ref{condition3}), and thus we have the map $\beta:S\times S\to S$ defined by $(s,t)\mapsto(s\backslash t)\backslash(s\vee t)$. Let $f:S\times S\to S^*$ be the map defined by $f(s,t)=\epsilon$ if $s=t$ and $f(s,t)=(ts)_{m_{st}-1}$ if $s\neq t$, where $\epsilon$ is the empty word. For $s\neq t$, since $s(ts)_{m_{st}-1}=(st)_{m_{st}}$ represents $s(s\backslash t)=s\vee t$, the word $(ts)_{m_{st}-1}$ must represent $s\backslash t$. Thus, $f$ is a right l.c.m. selector on $S$ in $\mathcal{M}$. For $s\neq t$ and $i=0,1,\ldots, m_{st}$, let $f_i(s,t)=s$ if $i$ is even and $f_i(s,t)=t$ if $i$ is odd. Then, for any $s\neq t$ and any $i=1,2,\ldots,m_{st}$, we have $f_i(s,t)= f_{i-1}(t,s)$ and $f(s,t)=f_1(s,t)f_2(s,t)\cdots f_{m_{st}-1}(s,t)$. Let $s\neq t$ be elements in $S$. Then
\begin{align*}
f(s,t)f_{m_{st}}(s,t)&=(ts)_{m_{st}-1}f_{m_{st}}(s,t)\\
&=(ts)_{m_{ts}-1}f_{m_{ts}-1}(t,s)\\
&=tf_1(t,s)f_2(t,s)\cdots f_{m_{ts}-1}(t,s)\\
&=tf(t,s)\\
&=t(st)_{m_{ts}-1}\\
&=(ts)_{m_{ts}}. 
\end{align*}
Since $(\mathcal{M},S)$ satisfies condition \eqref{condition2}, we have $(s\backslash t)\left((s\backslash t)\backslash(s\vee t)\right)=s\vee t$. Since $f(s,t)f_{m_{st}}(s,t)=(ts)_{m_{ts}}$ represents $(s\backslash t)\left((s\backslash t)\backslash(s\vee t)\right)=s\vee t$ and $f(s,t)$ represents $s\backslash t$, it follows that $f_{m_{st}}(s,t)$  represents $(s\backslash t)\backslash(s\vee t)$. Thus, we have $\beta(s,t)= f_{m_{st}}(s,t)$. By Corollary \ref{cor-present-garside-quandle-r}, the Artin quandle $\mathcal{D}(S^{\mathcal{A}})$ has a presentation with $S$ as its set of generators and defining relations as follows:
\begin{align}
&f_{m_{st}}(s,t)*f_{m_{st}-1}(s,t)*\cdots *f_1(s,t)=f_0(s,t)\qquad \text{and}\label{eqn26}\\
&u*f_{m_{st}-1}(s,t)*f_{m_{st}-2}(s,t)*\cdots *f_1(s,t)*f_0(s,t)\label{eqn27}\\
&=u*f_{m_{ts}-1}(t,s)*f_{m_{ts}-2}(t,s)*\cdots *f_1(t,s)*f_0(t,s)\nonumber
\end{align}
for $s,t,u\in S$ with $s\neq t$ (note that $s=f_0(s,t)$ and $t=f_0(t,s)$). We shall show that \eqref{eqn27} can be obtained by \eqref{eqn26}. Let $s,t,u\in S$ with $s\neq t$. Then
\begin{eqnarray*}
&& u*f_{m_{st}-1}(s,t)*f_{m_{st}-2}(s,t)*\cdots *f_1(s,t)*f_0(s,t)\\
&=& \left(u*f_{m_{st}-1}(s,t)*f_{m_{st}-2}(s,t)*\cdots *f_1(s,t)\right)*\left(f_{m_{st}}(s,t)*f_{m_{st}-1}(s,t)*\cdots *f_1(s,t)\right)\\
&=& u*f_{m_{st}}(s,t)*f_{m_{st}-1}(s,t)*\cdots *f_1(s,t)\\
&=& u*f_{m_{ts}-1}(t,s)*f_{m_{ts}-2}(t,s)*\cdots *f_1(t,s)*f_0(t,s),
\end{eqnarray*}
where the first equality is by \eqref{eqn26}, the second is by Lemma \ref{lem-2-present-dehn-quandle-right-gaussian-group} and third is by relations $f_i(s,t)=f_{i-1}(t,s)$. This shows that \eqref{eqn27} is redundant. Hence, if $\mathcal{A}$ is spherical, then the  Artin quandle $\mathcal{D}(S^{\mathcal{A}})$ has a presentation 
\begin{eqnarray}\label{presentation of artin using garside}
\mathcal{D}(S^{\mathcal{A}}) &=& \Bigg\langle S \mid (s*t)_{m_{st}}=s\quad \text{if}\; m_{st}\; \text{is even, and}\quad\\
\nonumber && (t*s)_{m_{st}}=s\quad \text{if}\; m_{st}\; \text{is odd for all}~ s\neq t \Bigg\rangle,
\end{eqnarray}
where  $(u*v)_k$ denotes the left associated product $u*v*u*v*u\cdots$ of length $k$.
\end{example}

\begin{example}
Let $M$ be the Garside monoid as in Example \ref{example1} and $G$ its group of fractions. The set $S=\{x_1,x_2,\ldots,x_n\}$ is the set of atoms in $M$ (see the proof of Proposition \ref{prop-dehornoy-paris-1}, i.e. \cite[Proposition 5.2]{Dehornoy-Paris-1999}). It can be seen that, for each $i\neq j$, the word $x_i^{p_i}$ represents $x_i\vee x_j$. Let $f(x_i,x_i)=\epsilon$ for $i=1,2,\ldots,n$ and $f(x_i,x_j)=x_i^{p_i-1}$ for $i\neq j$, where $\epsilon$ is the empty word. Then, $f$ is a right l.c.m. selector on $S$ in $M$. Let $\beta(x_i,x_j)=x_i$ for any $i$ and $j$. Then $x_if(x_i,x_j)=f(x_i,x_j)\beta(x_i,x_j)$ which, of course, represents $x_i\vee x_j$. In other words, $\beta(x_i,x_j)=(x_i\backslash x_j)\backslash(x_i\vee x_j)$. Then, by Corollary \ref{cor-present-garside-quandle-r}, the Garside quandle $\mathcal{D}(S^G)$ has a presentation 
$$\mathcal{D}(S^G)= \langle x_1,x_2,\ldots,x_n \mid x_i*^{p_j}x_j=x_i  \quad \textrm{for all} \quad  i\neq j  \rangle.$$
\end{example}

\begin{example} \label{example-torus-link-quandle}
Let $M$ be the Garside monoid as in Example \ref{example2} and $G$ its group of fractions. The set $S=\{x_1,x_2,\ldots,x_n\}$ is the set of atoms in $M$. For an integer $i$, let $x_i=x_j$ for $1\leq j\leq n$ such that $i\equiv j\,(\!\!\!\!\mod n)$. One can see that, for each $i\neq j$, the word $x_ix_{i+1}\cdots x_{i+m-1}$ represents $x_i\vee x_j$. Let $f(x_i,x_i)=\epsilon$ for $i=1,2,\ldots,n$ and $f(x_i,x_j)=x_{i+1}x_{i+2}\cdots x_{i+m-1}$ for $i\neq j$, where $\epsilon$ is the empty word. Then, $f$ is a right l.c.m. selector on $S$ in $M$. Let $\beta(x_i,x_i)=x_i$ for $i=1,2,\ldots,n$ and $\beta(x_i,x_j)=x_{i+m}$ for $i\neq j$. Then $x_if(x_i,x_j)=f(x_i,x_j)\beta(x_i,x_j)$, i.e. $\beta(x_i,x_j)=(x_i\backslash x_j)\backslash(x_i\vee x_j)$. Then, by Corollary \ref{cor-present-garside-quandle-r}, $\mathcal{D}(S^G)$ is generated by $S$ and has defining relations
\begin{eqnarray*}
x_{i+m}*x_{i+m-1}*\cdots*x_{i+1} &=& x_i,  \label{example 2.31 relation 1}\\
x_i*x_{j+m-1}*x_{j+m-2}*\cdots*x_j&=&x_i*x_{k+m-1}*x_{k+m-2}*\cdots*x_k \label{example 2.31 relation 2}
\end{eqnarray*}
for all $i,j,k=1,2,\ldots,n$ with $j \neq k$. Note that relations of second type are redundant. For, we have
\begin{align*}
x_i*x_{j+m-1}*x_{j+m-2}*\cdots*x_j&=(x_i*x_{j+m-1}*x_{j+m-2}*\cdots*x_{j+1})*(x_{j+m}*x_{j+m-1}*\cdots*x_{j+1})\\
&=x_i*x_{j+m}*x_{j+m-1}*\cdots*x_{j+1}
\end{align*}
for all $i,j=1,2,\ldots,n$, where the first equality follows from relations of first type and the second follows from Lemma \ref{lem-2-present-dehn-quandle-right-gaussian-group}. Iterating the process yields relations of second type, and hence the final presentation is
$$\mathcal{D}(S^G) = \langle x_1,x_2,\ldots,x_n \mid x_{i+m}*x_{i+m-1}*x_{i+m-2}*\cdots*x_{i+1}=x_i~ \textrm{for} ~i=1,2,\ldots,n \rangle.$$
\end{example}

\begin{remark}\label{torus link quandle}
It follows from \cite[Proposition 7.1]{MR4330281} that the link quandle $Q(T(n,m))$ of the torus link $T(n,m)$ has a presentation $$Q(T(n,m)) = \langle x_1,x_2,\ldots,x_n \mid x_{m+i}*x_m*x_{m-1}*\cdots*x_1=x_i~ \textrm{for} ~1\leq i\leq n \rangle,$$ where $x_l=x_k$ for $l\in\mathbb{Z}$ and $1\leq k\leq n$ such that $l\equiv k\pmod n$. Using induction on $j$ and relations $x_{m+i}*x_m*x_{m-1}*\cdots*x_1=x_i$, we can obtain relations $x_{m+i}*x_{m+j}*x_{m+j-1}*\cdots*x_{j+1}=x_i$ for $i,j=1,2,\ldots,n$. Further, it suffices to consider only $x_{m+i}*x_{m+i-1}*x_{m+i-2}*\cdots*x_{i+1}=x_i$  among latter relations. Thus, $Q(T(n,m))$ is isomorphic to the Garside quandle in Example \ref{example-torus-link-quandle}.
\end{remark}

\begin{example}
Let $M$ be the Garside monoid as in Example \ref{example3} and $G$ be its group of fractions. The set $S=\{x_1,x_2,y_1,y_2,y_3\}$ is the set of atoms in $M$. Let $p_1=2$ and $p_2=5$. For an integer $i$, let $y_i=y_j$ for $1\leq j\leq 3$ such that $i\equiv j\,(\!\!\!\!\mod 3)$. Using Corollary \ref{cor-present-garside-quandle-r} and after reducing the 
relations, we get a presentation 
\begin{eqnarray*}
\mathcal{D}(S^G) &=& \Bigg\langle x_1,x_2,y_1,y_2,y_3 \mid y_{k+4}*y_{k+3}*y_{k+2}*y_{k+1}=y_k, \quad x_i*^{p_j}x_j=x_i,\\
&&  x_i*y_1*y_3*y_2*y_1=x_i, \quad y_{k+4}*^{p_i}x_i=y_k ~\textrm{for}~ i,j=1,2~\textrm{with} ~i\neq j~\textrm{and}~k=1,2,3 \Bigg\rangle.
\end{eqnarray*}
\end{example}

\begin{example}
Let $M$ be the Garside monoid as in Example \ref{example4} and $G$ its group of fractions. The set $S=\{x_1,x_2,x_3,y_1,y_2,y_3\}$ is the set of atoms in $M$. Let $\Delta_1$ and $\Delta_2$ be minimal Garside elements in underling Artin monoids $M_1$ and $M_2$ of type $B_3$ and $A_3$, respectively. Then, the words $(x_1x_2x_3)^3$, $(x_2x_3x_1)^3$ and $(x_3x_1x_2)^3$ all represent $\Delta_1$, and the words $y_1y_2y_3y_1y_2y_1$, $(y_1y_3y_2)^2$, $(y_2y_1y_3)^2$ and $(y_3y_1y_2)^2$ all represent $\Delta_2$ (see \cite[examples 3 and 4]{Birman-Gebhardt-Meneses-2007} for the minimal Garside elements in Artin monoids of spherical type). Note that $M=(M_1 \hexstar M_2)/\equiv$, where $\equiv$ is the equivalence relation on $M_1  \hexstar M_2$ generated by $\Delta_1^2 =\Delta_2^3$. One can see that both $\Delta_1^2$ and $\Delta_2^3$ represent $x_i\vee y_j$ in $M$ for $i,j=1,2,3$. For an integer $i$, let $x_i=x_j$ and $y_i=y_j$ for $1\leq j\leq 3$ such that $i\equiv j\,(\!\!\!\!\mod 3)$. Define $f(x_i,y_j)=x_{i+1}x_{i+2}\cdots x_{i+17}$, $f(y_k,x_j)=y_{k+17}y_{k+16}\cdots y_{k+1}$ and $f(y_3,x_j)=y_1y_2\cdots y_{17}$ for $i,j=1,2,3$ and $k=1,2$. Then $f$ can be extended to a right l.c.m. selector on $S$ in $M$. For $i=1,2$, let $\psi_i$ be the permutation of the set of divisors of $\Delta_i$ given by $x\mapsto x\backslash\Delta_i$, and $\phi_i$ be the automorphism of $M_i$ defined by $\phi_i(x)=\psi_i^2(x)$ for a divisor $x$ of $\Delta_i$ (see \cite[lemmas 2.2 and 2.3]{Dehornoy2002}). One can verify that $x\psi_i(x)=\Delta_i=\psi_i(x)\phi_i(x)$ and $x\Delta_i=\Delta_i\phi_i(x)$ for a divisor $x$ of $\Delta_i$. Note that $\phi_i$ maps atoms to atoms. Let $z_1=y_3$, $z_2=y_2$ and $z_3=y_1$. It is easy to see that $\phi_1(x_i)=x_i$ and $\phi_2(y_i)=z_i$ for $i=1,2,3$. Thus $$x_i\psi_1(x_i)\Delta_1=\Delta_1^2=\psi_1(x_i)\Delta_1 x_i$$ and $$y_i\psi_2(y_i)\Delta_2^2=\Delta_2^3=\psi_2(y_i)\Delta_2^2z_i$$ for $i=1,2,3$. Define $\beta(x_i,y_j)=x_i$ and $\beta(y_j,x_i)=z_j$ for $i,j=1,2,3$. Then $\beta(x_i,y_j)=(\psi_1(x_i)\Delta_1)\backslash \Delta_1^2=(x_i\backslash y_j)\backslash(x_i\vee y_j)$ and $\beta(y_j,x_i)=(\psi_2(y_j)\Delta_2^2)\backslash \Delta_2^3=(y_j\backslash x_i)\backslash(y_j\vee x_i)$ for all $i,j=1,2,3$. Note that $\beta$ can be extended to the map $S\times S\to S$ defined as in \eqref{eqn2}. Using Corollary \ref{cor-present-garside-quandle-r} and after reducing the relations, we see that $\mathcal{D}(S^G)$ has a presentation 
\begin{eqnarray*}
&& \Bigg\langle x_1,x_2,x_3,y_1,y_2,y_3 \mid x_1*x_2*x_1*x_2=x_1*x_3=x_1,\quad x_2*x_1*x_2*x_1=x_3*x_2*x_3=x_2,\\
&& x_3*x_1=x_2*x_3*x_2= x_3, \quad y_2*y_1*y_2=y_1*y_3=y_1,\quad y_1*y_2*y_1=y_3*y_2*y_3=y_2,\\
&& y_3*y_1=y_2*y_3*y_2= y_3,\quad x_{i+18}*x_{i+17}*\cdots*x_{i+1}=x_i, \quad y_1*y_2*\cdots*y_{12}=y_1,\\
&& y_2*y_3*\cdots*y_{19}=y_2,\quad y_{12}*y_{11}*\cdots*y_1=y_3, \quad  x_i*y_2*y_3*\cdots*y_{19}=x_i,\\
&& y_1*x_{18}*x_{17}*\cdots *x_1=y_1*y_2*\cdots* y_7,\quad y_2*x_{18}*x_{17}*\cdots *x_1=y_2,\\
&& y_3*x_{18}*x_{17}*\cdots *x_1=y_1~\textrm{for}~ i=1,2,3 \Bigg\rangle.
\end{eqnarray*}
\end{example}

\medskip

\section{Presentations of Dehn quandles of groups}\label{section-present-dehn-quandles}
In this section, we prove a general result giving presentations of  Dehn quandles of groups when the centraliser of each generator is known. Although the result is general, determining generating sets for centralisers of elements in interesting classes of groups like Garside groups and Artin groups is usually challenging. See, for example, \cite{Birman-Gebhardt-Meneses-2007, Franco-Meneses-2003, MR2023190, Gebhardt-2005, MR0747249, MR2103472, Picantin2001b} for related works. Further, presentations obtained for Garside quandles using Theorem \ref{presentation-dehn-quandle} usually have larger number of relations than the one given by Theorem \ref{thm-present-dehn-quandle-right-gaussian-group}. We shall see many examples later in this section.

\subsection{Presentations of Dehn quandles} 
The following theorem gives a presentation of the Dehn quandle of a group $G$ with respect to a generating set $S$ when a generating set for the centraliser of each element in $S$ is known. 

\begin{theorem}\label{presentation-dehn-quandle}
Let $G$ be a group with a presentation $\langle S\mid R\rangle$. For $s \in S$, let $A_s$ be a generating set for the centraliser $\C_G(s)$ of $s$ in $G$. Let $T=\{(s,t)\in S\times S ~|~ s~\text{and}~ t~\text{are conjugate in}\;G\}$, and $f : T \to G$ a map such that 
\begin{itemize}
\item $f(t,s)sf(t,s)^{-1} = t$,
\item $f(t,s)= f(s,t)^{-1}$,
\item $f(u,t)f(t,s) = f(u,s)$
\end{itemize}
for all $(s,t), (t, u)\in T$. Then, the Dehn quandle $\dq(S^G)$ has a presentation with the generating set $S$ and defining relations as follows:
\begin{enumerate}[(1)]
\item For each $s \in S$ and each relation $r=s_{r_1}^{\delta_1}s_{r_2}^{\delta_2}  \cdots s_{r_k}^{\delta_k}$ in $R$, where $s_{r_i} \in S$ and $\delta_i \in   \{-1,1\}$, we have
$$s*^{\delta_{k}}s_{r_k}*^{\delta_{k-1}}s_{r_{k-1}}*^{\delta_{k-2}}\cdots *^{\delta{1}}s_{r_1}=s.$$ 

\item For each $s \in S$   and each $w=s_{w_1}^{\epsilon_1}s_{w_2}^{\epsilon_2}  \cdots s_{w_l}^{\epsilon_l}$ in $A_s$, where $s_{w_i} \in S$ and $\epsilon_i \in   \{-1,1\}$, we have
$$s*^{\epsilon_l}s_{w_l}*^{\epsilon_{l-1}}s_{w_{l-1}}*^{\epsilon_{l-2}}\cdots *^{\epsilon_1}s_{w_1}=s.$$

\item For each $(s, t )\in T$ such that $f(t, s)= f_1(t, s)^{\mu_1}f_2(t, s)^{\mu_2}\cdots f_n(t, s)^{\mu_n}$, where $f_i(t, s)\in S$ and $\mu_i\in\{-1,1\}$, we have 
$$s*^{\mu_n}f_n(t, s)*^{\mu_{n-1}} f_{n-1}(t, s)*^{\mu_{n-2}}\cdots*^{\mu_1}f _1(t, s)=t.$$ 
\end{enumerate}
\end{theorem}

\begin{proof}
Let $\psi: S \to \mathbf{S}$ be a bijection of $S$ onto another set $\mathbf{S}$, where, for brevity, we denote $\psi(s)$ by $\mathbf{s}$. Let $(Q,\star)$ be a quandle that has  a presentation with the set of generators $\mathbf{S}$ and defining relations as in (1), (2) and (3) written in terms of elements of $\mathbf{S}$. We write elements of $Q$ in bold to differentiate them from elements of $\dq(S^G)$. We claim that $\dq(S^G) \cong Q$. 
\par

Let $\phi : \mathbf{S} \to S$ given by $\phi(\mathbf{s})= s$ be the inverse of $\psi$. The map $\phi$ induces a quandle homomorphism $\phi:Q\to \dq(S^G)$, which is clearly surjective. It suffices to show that $\phi$ is injective.
\par

We first claim that if $x=s_{x_1}^{\eta_1}s_{x_2}^{\eta_2}\cdots s_{x_n}^{\eta_n}$ and $y=s_{y_1}^{\theta_1}s_{y_2}^{\theta_2}\cdots s_{y_m}^{\theta_m}$, where $s_{x_i}, s_{y_j}\in S$ and $\eta_i,\theta_j\in \{1,-1\}$, represent the same element of $G$, then 
$$\mathbf{s}\star^{\eta_n}\mathbf{s}_{x_n}\star^{\eta_{n-1}}\mathbf{s}_{x_{n-1}}\star^{\eta_{n-2}}\cdots\star^{\eta_1}\mathbf{s}_{x_1}=\mathbf{s}\star^{\theta_m}\mathbf{s}_{y_m}\star^{\theta_{m-1}}
\mathbf{s}_{y_{m-1}}\star^{\theta_{m-2}} \cdots \star^{\theta_1}\mathbf{s}_{y_1}$$ for all $\mathbf{s}\in\mathbf{S}$.

Since $x=y$ in $G$, there is a sequence of elements $x=g_0, g_1, g_2, \ldots, g_{\ell}=y$ in $G$ such that $g_i$ and $g_{i+1}$ differ by a single relation from $R\cup R'$, where $R'= \{s^{-1}s, ~ss^{-1} \mid s \in S\}$ is the set of trivial relations. Thus, it is enough to consider the case when $x$ and $y$ differ by a single relation i.e. when 
$x=s_{x_1}^{\eta_1}s_{x_2}^{\eta_2}\cdots s_{x_n}^{\eta_n}$ and $y=s_{x_1}^{\eta_1}s_{x_2}^{\eta_2}\cdots s_{x_i}^{\eta_i}r s_{x_{i+1}}^{\eta_{i+1}} \cdots  s_{x_n}^{\eta_n}$ for some $r\in R\cup R'$. If $r \in R'$, then the claim holds by the second quandle axiom (i.e. bijectivity of right multiplication). Suppose that $r=s_{r_{1}}^{\delta_1}s_{r_{2}}^{\delta_2}  \cdots s_{r_{k}}^{\delta_{k}}$ is an element of $R$. Then, Lemma \ref{lem-2-present-dehn-quandle-right-gaussian-group} gives
\begin{eqnarray*}
&& \mathbf{s}\star^{\eta_n}\mathbf{s}_{x_n}\star^{\eta_{n-1}} \cdots\star^{\eta_{i+1}}\mathbf{s}_{x_{i+1}}\star^{\delta_{k}}\mathbf{s}_{r_{k}}\star^{\delta_{k-1}}\mathbf{s}_{r_{k-1}}\star^{\delta_{k-2}}\cdots \star^{\delta{1}}\mathbf{s}_{r_1}\star^{\eta_i}\mathbf{s}_{x_{i}}\star^{\eta_{i-1}} \cdots\star^{\eta_1}\mathbf{s}_{x_1}\\
&=&(\mathbf{s}\star^{\delta_{k}}\mathbf{s}_{r_k}\star^{\delta_{k-1}}\mathbf{s}_{r_{k-1}}\star^{\delta_{k-2}}\cdots \star^{\delta{1}}\mathbf{s}_{r_1})\star^{\eta_n}(\mathbf{s}_{x_n}\star^{\delta_{k}}\mathbf{s}_{r_k}\star^{\delta_{k-1}}\mathbf{s}_{r_{k-1}}\star^{\delta_{k-2}}\cdots \star^{\delta{1}}\mathbf{s}_{r_1})\star^{\eta_{n-1}} \cdots\star^{\eta_{i+1}}\\
&& (\mathbf{s}_{x_{i+1}}\star^{\delta_{k}}\mathbf{s}_{r_k}\star^{\delta_{k-1}}\mathbf{s}_{r_{k-1}}\star^{\delta_{k-2}}\cdots \star^{\delta{1}}\mathbf{s}_{r_1})\star^{\eta_i}\mathbf{s}_{x_{i}}\cdots\star^{\eta_1}\mathbf{s}_{x_1}\\
&=& \mathbf{s}\star^{\eta_n}\mathbf{s}_{x_n}\star^{\eta_{n-1}}\cdots\star^{\eta_{i+1}}\mathbf{s}_{x_{i+1}}\star^{\eta_i}\mathbf{s}_{x_{i}} \star^{\eta_{i-1}}\cdots\star^{\eta_1}\mathbf{s}_{x_1}, \quad \textrm{by relation of type (1)}.
\end{eqnarray*}
This proves the claim.
\medskip

Now, to prove injectivity of $\phi$, take two elements $\mathbf{s}\star^{\alpha_p} \mathbf{s}_{a_p}\star^{\alpha_{p-1}}\mathbf{s}_{a_{p-1}}\star^{\alpha_{p-2}} \cdots\star^{\alpha_1}\mathbf{s}_{a_1}$ and $\mathbf{t}\star^{\beta_q} \mathbf{s}_{b_q}\star^{\beta_{q-1}}\mathbf{s}_{b_{q-1}}\star^{\beta_{q-2}}\cdots\star^{\beta_1}\mathbf{s}_{b_1}$ in $Q$ such that 
\begin{equation}\label{eq1}
s*^{\alpha_p}s_{a_p}*^{\alpha_{p-1}}s_{a_{p-1}}*^{\alpha_{p-2}}\cdots*^{\alpha_1}s_{a_1}=t*^{\beta_q}s_{b_q}*^{\beta_{q-1}}s_{b_{q-1}}*^{\beta_{q-2}}\cdots*^{\beta_1}s_{b_1}
\end{equation}
 in $\dq(S^G)$. Rewriting \eqref{eq1} gives  
 \begin{equation}\label{eq2}
 s*^{\alpha_p}s_{a_p}*^{\alpha_{p-1}}s_{a_{p-1}}*^{\alpha_{p-2}}\cdots*^{\alpha_1}s_{a_1}*^{-\beta_1}s_{b_1}*^{-\beta_2}s_{b_2}*^{-\beta_3} \cdots*^{-\beta_q}s_{b_q}=t.
 \end{equation}
Since $s$ and $t$ are conjugate in $G$,  we have $(s,t)\in T$. Thus, we have
 \begin{equation}\label{eq3}
s*^{\mu_n}f_{n}(t, s)*^{\mu_{n-1}}f_{n-1}(t, s)*^{\mu_{n-2}}\cdots *^{\mu_1}f_1(t, s)=t,
 \end{equation}
where $f(t, s) \in G$ such that $f(t, s)= f_1(t, s)^{\mu_1}f_2(t, s)^{\mu_2}\cdots f_n(t, s)^{\mu_n}$ for $f_i(t, s)\in S$ and $\mu_i\in\{-1,1\}$. Using \eqref{eq2} and \eqref{eq3}, we have 
\begin{eqnarray}\label{eq4}
\nonumber s &=& s*^{\alpha_p}s_{a_p}*^{\alpha_{p-1}}s_{a_{p-1}}*^{\alpha_{p-2}} \cdots*^{\alpha_1}s_{a_1}*^{-\beta_1}s_{b_1}*^{-\beta_2}s_{b_2}*^{-\beta_3}\cdots*^{-\beta_q}s_{b_q}\\
& & *^{-\mu_1}f_1(t, s)*^{-\mu_2}f_2(t, s)*^{-\mu_3}\cdots*^{-\mu_n}f_n(t, s).
 \end{eqnarray}

Writing \eqref{eq4} in terms of conjugation in $G$ implies that the element 
$$f_n(t, s)^{-\mu_n}f_{n-1}(t, s)^{-\mu_{n-1}}\cdots f_1(t, s)^{-\mu_1}s_{b_q}^{-\beta_q}s_{b_{q-1}}^{-\beta_{q-1}}\cdots s_{b_1}^{-\beta_1}s_{a_1}^{\alpha_1}s_{a_2}^{\alpha_2}\cdots s_{a_p}^{\alpha_p} $$
 commutes with $s$, and hence equals to an element, say, $w$ of $\C_G(s)$. Without loss of generality, we can assume that $w \in A_s$. If not, then $w$ will be a product of elements of $A_s$ and the argument will be similar. Let $w=s_{w_1}^{\partial_1}s_{w_2}^{\partial_2}  \cdots s_{w_l}^{\partial_l}$ written in terms of generators $S$ of $G$. Then the equality
 $$s_{w_1}^{\partial_1}s_{w_2}^{\partial_2}  \cdots s_{w_l}^{\partial_l}= f_n(t, s)^{-\mu_n}f_{n-1}(t, s)^{-\mu_{n-1}}\cdots f_1(t, s)^{-\mu_1}s_{b_q}^{-\beta_q}s_{b_{q-1}}^{-\beta_{q-1}}\cdots s_{b_1}^{-\beta_1}s_{a_1}^{\alpha_1}s_{a_2}^{\alpha_2}\cdots s_{a_p}^{\alpha_p}  $$
of elements in $G$ can be rewritten as
\begin{equation}\label{eq5}
s_{b_1}^{\beta_1}s_{b_2}^{\beta_2}\cdots s_{b_q}^{\beta_q}f_1(t, s)^{\mu_1}f_2(t, s)^{\mu_2}\cdots f_n(t, s)^{\mu_n}s_{w_1}^{\partial_1}s_{w_2}^{\partial_2}  \cdots s_{w_l}^{\partial_l}= s_{a_1}^{\alpha_1}s_{a_2}^{\alpha_2}\cdots s_{a_p}^{\alpha_p}.
 \end{equation}
Our earlier proved claim gives
\begin{eqnarray*}
& & \mathbf{s}\star^{\alpha_p} \mathbf{s}_{a_p}\star^{\alpha_{p-1}}\mathbf{s}_{a_{p-1}}\star^{\alpha_{p-2}} \cdots\star^{\alpha_1}\mathbf{s}_{a_1}\\
&=& \mathbf{s}\star^{\partial_l}\mathbf{s}_{w_l}\star^{\partial_{l-1}}\mathbf{s}_{w_{l-1}}\star^{\partial_{l-2}} \cdots \star^{\partial_1}\mathbf{s}_{w_1}\star^{\mu_n}\mathbf{f}_n\mathbf{(t, s)}\star^{\mu_{n-1}}\mathbf{f}_{n-1}\mathbf{(t,s)}\star^{\mu_{n-2}}\cdots\star^{\mu_1}\mathbf{f}_1\mathbf{(t,s)}\\
&& \star^{\beta_{q}}\mathbf{s}_{b_q}\star^{\beta_{q-1}}\mathbf{s}_{b_{q-1}}\star^{\beta_{q-2}}\cdots\star^{\beta_1}\mathbf{s}_{b_1}\\
&=&  \mathbf{s} \star^{\mu_n}\mathbf{f}_n\mathbf{(t, s)}\star^{\mu_{n-1}}\mathbf{f}_{n-1}\mathbf{(t,s)}\star^{\mu_{n-2}}\cdots\star^{\mu_1}\mathbf{f}_1\mathbf{(t,s)} \star^{\beta_{q}}\mathbf{s}_{b_q}\star^{\beta_{q-1}}\mathbf{s}_{b_{q-1}}\star^{\beta_{q-2}} \cdots\star^{\beta_1}\mathbf{s}_{b_1},\\
&& \textrm{by relation of type (2)}\\
&=&\mathbf{t}\star^{\beta_{q}}\mathbf{s}_{b_q}\star^{\beta_{q-1}}\mathbf{s}_{b_{q-1}}\star^{\beta_{q-2}}\cdots\star^{\beta_1}\mathbf{s}_{b_1}, \quad \textrm{by relation of type (3)}.
\end{eqnarray*}
This completes the proof of the theorem.
\end{proof}

\begin{corollary}\label{finitely gen dehn quandle}
If $G=\langle S \mid R \rangle$ is a finitely presented group such that the centraliser of each generator from $S$ is finitely generated, then the Dehn quandle $D(S^G)$ is finitely presented.
\end{corollary}

It follows from \cite[Theorem 12]{MR2023190} that the centraliser of an element in a Garside group is finitely generated. Thus, the corollary holds for Garside groups, which includes spherical Artin groups.

\begin{remark}
One can reduce the number of relations in Theorem \ref{presentation-dehn-quandle} as follows:
\begin{enumerate}
\item Relations of type $(2)$ need to be checked only for those generators in $S$ that represent distinct conjugacy classes in $G$.
\item We write $S= \sqcup_{i} X_i$ as a disjoint union of sets where $X_i$ consists of elements that are conjugate to each other. We equip each $X_i$ with a partial order such that the corresponding poset graph is a tree. Then, for each $i$, relations of type $(3)$ need to be checked only for elements in $X_i \times X_i$ that are adjacent in the corresponding poset graph.
\end{enumerate}
\end{remark}
\medskip

\subsection{Examples of presentations of Dehn quandles}
As applications of Theorem \ref{presentation-dehn-quandle}, we give several examples of presentations of Dehn quandles. 

\begin{example}
Let $\Core{(\mathbb{Z}_n)}$ be the core quandle of the cyclic group of order $n$ (also called the dihedral quandle of order $n$). Then, the quandle operation in $\Core{(\mathbb{Z}_n)}$ is given by $x*y=2y-x \mod n$. It is shown in \cite[Proposition 3.11]{Dhanwani-Raundal-Singh-2021} that the core quandle of any group is  a Dehn quandle of some group. Let $\text{D}_n= \langle s_1, s_2 \mid s_1^2=1, s_2^2=1, (s_1s_2)^n=1 \rangle $ be the Coxeter presentation of the dihedral group of order $2n$, and set $S=\{s_1,s_2\}$. It follows from a direct computation (see also \cite[Lemma 6.3]{BardakovNasybullov2020}) that $\Core(\mathbb{Z}_n)\cong \dq(S^{\text{D}_n})$.
\medskip

First consider the case when $n$ is odd. In this case, $s_1$ and $s_2$ are conjugate, and $\C_{\text{D}_n}(s_i)=\langle s_i \rangle$.  Set $s_1<s_2$ and $f(s_2,s_1)={(s_1s_2)}^{\frac{n-1}{2}}$. Let $(s_1*s_2)_n$ denote the left associated product $s_1*s_2*s_1*s_2*\cdots$ of length $n$. Then Theorem~\ref{presentation-dehn-quandle} gives the following set of relations:
\begin{eqnarray}
s_i*s_j*s_j &= & s_i \quad \text{for}~i, j=1,2,\label{(2.2.1)}\\
s_i*\underbrace{s_2*s_1*s_2*s_1*\cdots *s_2*s_1}_{2n~ \textrm{terms}} &=& s_i  \quad \text{for}~ i=1,2,\label{(2.2.2)}\\
(s_1*s_2)_{n}&=&s_2 \label{(2.2.4)}. 
\end{eqnarray}
Note that \eqref{(2.2.1)} for $i=j$ are trivial relations. Since $n$ is odd, using  \eqref{(2.2.1)}, the relation \eqref{(2.2.4)} can be turned into the relation $(s_2*s_1)_n=s_1$. Now, using \eqref{(2.2.4)} and $(s_2*s_1)_n=s_1$, we can recover \eqref{(2.2.2)}. Thus, the final presentation is  
$$\Core(\mathbb{Z}_n) \cong \langle s_1,s_2~\mid~ s_2*s_1*s_1=s_2, \quad s_1*s_2*s_2=s_1, \quad (s_1*s_2)_{n}=s_2 \rangle.$$ 

Suppose now that $n$ is even. In this case, $s_1$ and $s_2$ lie in different conjugacy classes. Also, $\C_{\text{D}_n}(s_i)=\langle s_i,(s_1s_2)^{\frac{n}{2}}\rangle$. Again, Theorem~\ref{presentation-dehn-quandle} gives the following set of relations:
\begin{eqnarray}
s_i*s_j*s_j &= & s_i \quad \text{for}~i, j=1,2,\label{(2.2.5)}\\
s_i*\underbrace{s_2*s_1*s_2*s_1*\cdots *s_2*s_1}_{2n~ \textrm{{terms}}} &=& s_i  \quad \text{for}~ i=1,2,\label{(2.2.6)}\\
s_1*\underbrace{s_2*s_1*s_2*s_1*\cdots *s_2 * s_1}_{n~ \textrm{{terms}}} &=& s_1,\label{(2.2.7)}\\
s_2*\underbrace{s_2*s_1*s_2*s_1*\cdots *s_2 * s_1}_{n~ \textrm{{terms}}} &=& s_2\label{(2.2.8)}.
\end{eqnarray}
Again we ignore the trivial relations. Further, using \eqref{(2.2.7)} and \eqref{(2.2.8)}, we can recover relations in \eqref{(2.2.6)}. Hence, the final presentation is 
$$\Core(\mathbb{Z}_n)\cong \langle s_1,s_2~\mid~ s_2*s_1*s_1=s_2,\quad  s_1*s_2*s_2=s_1, \quad (s_1*s_2)_{n}=s_1, \quad (s_2*s_1)_{n}=s_2 \rangle.$$ 
\end{example}

\begin{remark}
If $n$ is even, it follows from the presentation of $Q(T(n,m))$ as given by Remark \ref{torus link quandle} and the presentation of $\Core(\mathbb{Z}_m)$ that the map  $x_i\to s_j$ for $i\equiv j \pmod 2$ defines a surjective quandle homomorphism from $Q(T(n,m))$ onto $\Core(\mathbb{Z}_m)$. Thus, for $n$ even, the torus link $T(n,m)$ is always $m$-colorable. This observation seems to be well-known. 
\end{remark}

\begin{example}
The fundamental group of a closed orientable surface $S_g$ of genus $g \ge 2$ has a presentation
$$\pi_1(S_g) =\langle a_1,b_1,a_2,b_2, \ldots,a_g,b_g~ \mid~[a_1,b_1][a_2,b_2]\cdots[a_{g},b_{g}]=1\rangle.$$ Setting $S=\{a_1,b_1,a_2,b_2, \ldots,a_g,b_g\}$, we note that no two elements of $S$ are conjugate to each other. Further, it is known that $\C_{\pi_1(S_g)}(a_i)=\langle a_i \rangle$ and $\C_{\pi_1(S_g)}(b_i)=\langle b_i \rangle$ \cite[Section 1.1.3]{Farb-Margalit2012}. Thus, by Theorem~\ref{presentation-dehn-quandle}, we get
\begin{eqnarray}\label{surface group dehn quandle}
\nonumber \dq(S^{\pi_1(S_g)}) &=& \Bigg\langle a_1,b_1,a_2,b_2, \ldots,a_g,b_g~|~a_i*b_{g}*a_{g}*^{-1}b_{g}*^{-1}a_{g}*\cdots*b_1*a_1*^{-1}b_1*^{-1}a_1=a_i,\\
&& b_i*b_{g}*a_{g}*^{-1}b_{g}*^{-1}a_{g}*\cdots*b_1*a_1*^{-1}b_1*^{-1}a_1=b_i \quad  \textrm{for all}~i \Bigg\rangle.
\end{eqnarray}
We note that the presentation of the enveloping group of $\dq(S^{\pi_1(S_g)})$ given in \cite[Theorem 3.18]{Dhanwani-Raundal-Singh-2021} can be recovered using \eqref{surface group dehn quandle} and
\cite[Theorem 5.1.7]{Winker1984}.
\end{example}
\medskip

\begin{example}
For $n \ge 3$, recall the Artin presentation of the braid group 
$$B_n=\langle \sigma_1,\sigma_2,\ldots,\sigma_{n-1}~ \mid \sigma_i\sigma_j=\sigma_j\sigma_i \text{ for } |i-j|\geq 2, \quad \sigma_i\sigma_{i+1}\sigma_{i}=\sigma_{i+1}\sigma_{i}\sigma_{i+1} \text{ for all } i \rangle.$$ Let us set $S=\{\sigma_1,\sigma_2,\ldots,\sigma_{n-1} \}$ and the total ordering $\sigma_1 <\sigma_2< \cdots <\sigma_{n-1}$ on $S$. Note that all the elements of $S$ are conjugate to each other. Choose $f(\sigma_{i+1},\sigma_i)=\sigma_i\sigma_{i+1}$ for each $1\leq i\leq n-2$. It follows from \cite[Theorem 4]{MR0747249} that $\C_{B_n}(\sigma_1)$ is generated by $X \cup Y$, where
\begin{eqnarray*}
X &=& \{ \sigma_1,\sigma_3, \sigma_4, \ldots,  \sigma_{n-1},\sigma_2\sigma_1\sigma_1\sigma_2 \}\\
\textrm{and} && \\
Y &=& \{ \sigma_1^2, (\sigma_2\sigma_1)(\sigma_3\sigma_2)\cdots(\sigma_{r+1}\sigma_r)(\sigma_r\sigma_{r+1})\cdots(\sigma_2\sigma_3)(\sigma_1\sigma_2) \quad \text{for} \quad 1\leq r\leq n-2 \}.
\end{eqnarray*}
It follows from the discussion after \cite[Theorem 4]{MR0747249} that generators from $Y$ can be written in terms of generators from $X$. In fact, setting
$x_r= (\sigma_2\sigma_1)(\sigma_3\sigma_2)\cdots(\sigma_{r+1}\sigma_r)(\sigma_r\sigma_{r+1})\cdots(\sigma_2\sigma_3)(\sigma_1\sigma_2),$ one can show that $x_r=x_{r-1}\sigma_{r+1}\sigma_{r}\cdots\sigma_{4}\sigma_{3}x_1\sigma_{3}^{-1}\sigma_{4}^{-1}\cdots\sigma_{r}^{-1}\sigma_{r+1}^{-1}$ for each $r\geq2$. Thus, we obtain 
\begin{equation}\label{centraliser-bn}
\C_{B_n}(\sigma_1)=\langle \sigma_1,\sigma_3, \sigma_4, \ldots,  \sigma_{n-1},\sigma_2\sigma_1\sigma_1\sigma_2 \rangle.
\end{equation}
\par

First we consider $n=3$ to clarify the general idea. We have $f(\sigma_{2},\sigma_1)=\sigma_1\sigma_{2}$ and $\C_{B_3}(\sigma_1)=\langle \sigma_1, \sigma_2\sigma_1\sigma_1\sigma_2\rangle$. Thus, by Theorem~\ref{presentation-dehn-quandle}, the defining relations in $\dq(S^{B_3})$ are given by
\begin{eqnarray}
\sigma_i*\sigma_1*\sigma_2*\sigma_1*^{-1}\sigma_2*^{-1}\sigma_1*^{-1}\sigma_2 &=&\sigma_i \quad \text{for}~ i=1,2,\label{(2.2.11)}\\
\sigma_1*\sigma_2*\sigma_1*\sigma_1*\sigma_2&=&\sigma_1,\label{(2.2.13)}\\
\sigma_1*\sigma_2*\sigma_1&=& \sigma_2 \label{(2.2.14)}. 
\end{eqnarray}
By using \eqref{(2.2.14)} in  \eqref{(2.2.13)}, we get
\begin{equation}\label{(2.2.15)}
\sigma_2*\sigma_1*\sigma_2=\sigma_1.
\end{equation}
It follows from relations \eqref{(2.2.14)}  and \eqref{(2.2.15)} that \eqref{(2.2.11)} is redundant. Thus, we obtain
$$\dq(S^{B_3})=\langle \sigma_1,\sigma_2~\mid~\sigma_1*\sigma_2*\sigma_1=\sigma_2, \quad \sigma_2*\sigma_1*\sigma_2=\sigma_1 \rangle.$$
\par

Now, assume that $n \ge 4$. In this case, defining relations for $\dq(S^{B_n})$ are as follows:
\begin{eqnarray}
\sigma_k*\sigma_i*\sigma_{i+1}*\sigma_i &= & \sigma_k \quad \text{for } 1\leq k\leq n-1~ \text{and }~ 1\leq i\leq n-2, \label{Type 1a}\\
\nonumber *^{-1}\sigma_{i+1}*^{-1}\sigma_i*^{-1}\sigma_{i+1} &&\\
\sigma_k*\sigma_i*\sigma_{j}*^{-1}\sigma_i*^{-1}\sigma_j &= & \sigma_k \quad \text{for } 1\leq k\leq n-1~ \text{and } ~ 3\leq i+2 \le j \leq n-1,\label{Type 1b}\\
\sigma_1*\sigma_j &= & \sigma_1 \quad \text{for }  3\leq j\leq n-1,\label{Type 2a}\\
\sigma_1*\sigma_2*\sigma_1*\sigma_1*\sigma_2 &= & \sigma_1, \label{Type 2b}\\
\sigma_i*\sigma_{i+1}*\sigma_i &= & \sigma_{i+1} \quad \text{for }  1\leq i\leq n-2\label{Type 3}.
\end{eqnarray}

Using relations in \eqref{Type 3}, we can recover relations in \eqref{Type 1a}, and the relation \eqref{Type 2b} can be rewritten as $\sigma_2*\sigma_1*\sigma_2 =  \sigma_1.$ Taking $k=i-1$ and using induction on $n$, we can rewrite relations in  \eqref{Type 1b} as $\sigma_i*\sigma_j=\sigma_i$ for $3\leq i+2 \le j \leq n-1$. The remaining relations in \eqref{Type 1b} can be recovered using these relations.  Thus, the presentation $\mathcal{D}(S^{B_n})$ is
\begin{eqnarray}\label{presentation of bn using thm2.1}
\mathcal{D}(S^{B_n}) &=& \Bigg\langle \sigma_1, \ldots, \sigma_{n-1} \mid \sigma_i*\sigma_{i+1}*\sigma_i =  \sigma_{i+1}  \quad \text{for } 1\leq i\leq n-2,\\ 
\nonumber &&  \sigma_2*\sigma_1*\sigma_2=\sigma_1,  \quad \sigma_i*\sigma_j=\sigma_i \quad \text{ for } 3 \le  i+2\leq j\leq n-1 \Bigg\rangle.
\end{eqnarray}
\end{example}

\begin{remark}
Using induction on $n$, one can show that the presentation for $\dq(S^{B_n})$ given by \eqref{presentation of artin using garside}  can also be reduced to the one given in \eqref{presentation of bn using thm2.1}.
\end{remark}
\medskip

\begin{example}\label{raag quandle}
Recall the presentation of an Artin group from \eqref{presentation of artin group}. We say that  $\mathcal{A}$ is a {\it right angled Artin group} if each $m_{ij}$ is either $2$ or $\infty$. Setting $S=\{x_1, x_2, \dots, x_n\}$ to be the Artin generating set of $\mathcal{A}$, we see that no two elements from $S$ are conjugate to each other. Further, $\C_{\mathcal{A}}(x_i)= \langle x_j \mid  m_{ij}=2\rangle$. Thus, by Theorem~\ref{presentation-dehn-quandle}, $\dq(S^{\mathcal{A}})$  has a presentation with generating set $S$ and defining relations as
\begin{eqnarray}
 x_k*x_i*x_j*^{-1} x_i *^{-1} x_j&=& x_k \quad \text{for}~1\leq k\leq n~ \text{and}~ m_{ij}=2, \label{(3.2.25)}\\
 x_i*x_j&=& x_i \quad \text{for}~ m_{ij}=2\label{(3.2.26)}. 
\end{eqnarray}
Clearly, relations in \eqref{(3.2.25)} can be recovered from  relations in \eqref{(3.2.26)}. Thus, the presentation is 
\begin{equation}\label{presentation raag quandle}
\dq(S^{\mathcal{A}})= \langle x_1, x_2, \dots, x_n ~\mid~ x_i*x_j=x_i \quad \text{whenever} \quad m_{ij}=2\rangle.
\end{equation}
\end{example}

In view of examples \ref{spherical artin} and \ref{raag quandle}, we propose the following.

\begin{conjecture}
If $\mathcal{A}=\langle S~|~ (st)_{m_{st}}=(ts)_{m_{ts}}\; \text{for}\; s,t\in S\; \text{with}\; s\neq t \rangle$ is an Artin group, then its Dehn quandle $\mathcal{D}(S^{\mathcal{A}})$ has a presentation
$$\mathcal{D}(S^{\mathcal{A}}) = \langle S \mid (s*t)_{m_{st}}=s~ \text{if}\; m_{st}~ \text{is even, and}~  (t*s)_{m_{st}}=s~ \text{if}~ m_{st}~ \text{is odd for all}~ s\neq t \rangle.$$
\end{conjecture}
\medskip

Next, we consider Dehn quandles of surfaces. Let $S_{g,p}$ be an orientable surface of genus $g$ with $p$ marked points. The {\it mapping class group} $\mathcal{M}_{g,p}$ of $S_{g,p}$ is defined as the set of isotopy classes of orientation preserving self-homeomorphisms of $S_{g,p}$ which permute the set of marked points. A simple closed curve $\alpha$ on $S_{g,p}$ is said to be \textit{non-separating} if $S_{g,p}\setminus \alpha$ is connected, and called \textit{separating} otherwise. Separating and non-separating simple closed arcs are defined analogously. Let $\dq_{g,p}^{ns}$ be the set of isotopy classes of all non-separating simple closed curves and simple closed arcs in $S_{g,p}$. For each isotopy class $y \in \dq_{g,p}^{ns}$ of a non-separating simple closed curve, let $T_y$  be the isotopy class of the right hand {\it Dehn twist} along any non-separating simple closed curve representing $y$. Similarly, if $y$ is the isotopy class of a non-separating simple closed arc, then we  define $H_y$ to be the isotopy class of the anti-clockwise {\it half twist} about the punctures joined by any arc representing $y$. This defines an injective map
$$ \tau: \dq_{g,p}^{ns} \hookrightarrow \mathcal{M}_{g,p}$$
by setting
$$
\tau(x)=\begin{cases}
T_x \text{ if $x$ is the isotopy class of a non-separating simple closed curve,}\\
H_x \text{ if $x$ is the isotopy class of a non-separating simple closed arc.}
\end{cases}
$$
For $x, y \in \dq_{g,p}^{ns}$, defining $x*y= \tau(y)(x)$ turns $\dq_{g,p}^{ns}$ into a quandle (see \cite{Dhanwani-Raundal-Singh-2021} for details), called the Dehn quandle of the surface $S_{g,p}$.  Note that the set of all Dehn twists along non-separating simple closed curves forms one conjugacy class in $\mcg_{g,p}$, whereas the set of all half twists along simple closed arcs forms another conjugacy class. For each $g,p \ge 0$, the group  $\mcg_{g,p}$ is generated by finitely many Dehn twists about non-separating simple closed curves and half twists about simple closed arcs \cite[Corollary 4.15]{Farb-Margalit2012}. Thus, if $S$ is such a generating set for $\mcg_{g,p}$, then  $\tau: \dq_{g,p}^{ns} \to \dq(S^{\mcg_{g,p}})$ becomes an isomorphism of quandles since
$$\tau(x * y)=\tau(\tau(y)(x))= \tau(y) \tau(x) \tau(y)^{-1}= \tau(x) * \tau(y).$$
Hence, $\dq_{g,p}^{ns}$ has the structure of the Dehn quandle of the group $\mcg_{g,p}$ with respect to $S$.
\par
The construction of Dehn quandle of $S_{g,p}$ for $p=0$ first appeared in the work of Zablow \cite{Zablow1999, Zablow2003}. Further, \cite{kamadamatsumoto,Yetter2003} considered the quandle structure on the set of isotopy classes of simple closed arcs in $S_{g,p}$ for $p\geq 2$, and called it {\it quandle of cords}. In general, the quandle of cords is a subquandle of $\mathcal{D}_{g,p}^{ns}$. In the case of a disk with $n$ marked points, this quandle can be identified with the Dehn quandle of the braid group $B_n$ with respect to its standard set of generators, that is, half twists along the cords. The reader may refer to \cite[Section 5]{Dhanwani-Raundal-Singh-2021} for generalities on Dehn quandles of surfaces, and \cite{Farb-Margalit2012} for basic facts about mapping class groups. 

\begin{example}
Recall from \cite[Section 5.1.3]{Farb-Margalit2012} that a presentation of the mapping class group $\mathcal{M}_1$ of the closed orientable surface of genus one (namely, the torus) is given by
$$\mathcal{M}_1=\langle T_{a}, T_{b}~|~T_{a}T_{b}T_{a}=T_{b}T_{a}T_{b}, \quad (T_{a}T_{b})^6=1\rangle,$$ where the curves $a,b$ are as shown in Figure \ref{Humphries-curves}(A). 

\begin{figure}[hbt!]
\begin{subfigure}{0.386\textwidth}
\centering
\begin{tikzpicture}[scale=0.5]
\begin{knot}[clip width=6, clip radius=4pt]
\strand[-] (0,4) to [out=left, in=left, looseness=2.3] (0,-2.9) to [out=right, in=right, looseness=2.3] (0,4);
\end{knot};
\begin{scope}
\clip (-1,0.5)rectangle(1,1);
\draw (0,0) circle [x radius=9mm, y radius=9mm];
\end{scope}
\begin{scope}[shift={(0,1.7)}]
\clip (-1.1,-0.8)rectangle(1.1,-1.5);
\draw (0,0) circle [x radius=14mm, y radius=14mm];
\end{scope}
\begin{scope}
\clip (0,0.5)rectangle(-0.6,-3.2);
\draw (0,-1.33) circle [x radius=4mm, y radius=16.1mm];
\end{scope}
\begin{scope}
\clip (0,0.5)rectangle(0.6,-3.2);
\draw[dashed] (0,-1.33) circle [x radius=4mm, y radius=16.1mm];
\end{scope}
\draw (0,0.6) circle [x radius=20mm, y radius=14mm];
\node at (-0.93,-1.7) {$a$}; \node at (0,2.5) {$b$};
\end{tikzpicture}
\caption{Genus one}
\label{genus-one}
\end{subfigure}
\begin{subfigure}{0.6\textwidth}
\centering
\begin{tikzpicture}[scale=0.5]
\begin{knot}[clip width=6, clip radius=4pt]
\strand[-] (0,4) to [out=left, in=left, looseness=2.3] (0,-2.9) to [out=right, in=left, looseness=1.3] (3,-2.3) to [out=right, in=left, looseness=1.3] (6,-2.9) to [out=right, in=right, looseness=2.3] (6,4) to [out=left, in=right, looseness=1.3] (3,3.4) to [out=left, in=right, looseness=1.3] (0,4);
\end{knot};

\begin{scope}
\clip (-1,0.5)rectangle(1,1);
\draw (0,0) circle [x radius=9mm, y radius=9mm];
\end{scope}
\begin{scope}[shift={(0,1.7)}]
\clip (-1.1,-0.8)rectangle(1.1,-1.5);
\draw (0,0) circle [x radius=14mm, y radius=14mm];
\end{scope}
\begin{scope}[shift={(6,0)}]
\clip (-1,0.5)rectangle(1,1);
\draw (0,0) circle [x radius=9mm, y radius=9mm];
\end{scope}
\begin{scope}[shift={(6,1.7)}]
\clip (-1.1,-0.8)rectangle(1.1,-1.5);
\draw (0,0) circle [x radius=14mm, y radius=14mm];
\end{scope}

\draw (0,0.6) circle [x radius=20mm, y radius=14mm];
\draw (6,0.6) circle [x radius=20mm, y radius=14mm];

\begin{scope}
\clip (0,0.5)rectangle(-0.6,-3.2);
\draw (0,-1.33) circle [x radius=4mm, y radius=16.1mm];
\end{scope}
\begin{scope}
\clip (0,0.5)rectangle(0.6,-3.2);
\draw[dashed] (0,-1.33) circle [x radius=4mm, y radius=16.1mm];
\end{scope}
\begin{scope}[shift={(6,0)}]
\clip (0,0.5)rectangle(-0.6,-3.2);
\draw (0,-1.33) circle [x radius=4mm, y radius=16.1mm];
\end{scope}
\begin{scope}[shift={(6,0)}]
\clip (0,0.5)rectangle(0.6,-3.2);
\draw[dashed] (0,-1.33) circle [x radius=4mm, y radius=16.1mm];
\end{scope}

\begin{scope}
\clip (0.7,0)rectangle(5.3,0.5);
\draw[dashed] (3,0.5) circle [x radius=23mm, y radius=3mm];
\end{scope}
\begin{scope}
\clip (0.7,0.5)rectangle(5.3,1);
\draw (3,0.5) circle [x radius=23mm, y radius=3mm];
\end{scope}

\node at (-0.93,-1.7) {$a_1$}; \node at (0,2.5) {$b_1$}; \node at (3,1.27) {$c_1$}; \node at (6,2.5) {$b_2$}; \node at (5.02,-1.7) {$a_2$};
\end{tikzpicture}
\caption{Genus two}
\label{genus-two}
\end{subfigure}
\caption{Humphries curves on a surface}
\label{Humphries-curves}
\end{figure}
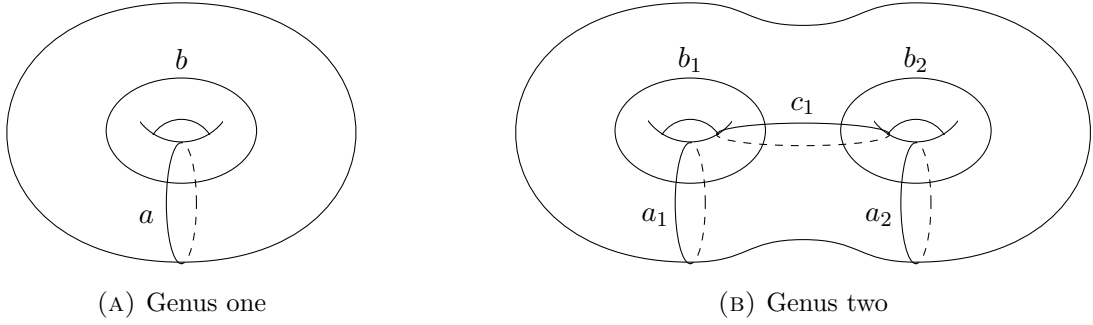

Let us set $S=\{T_{a}, T_{b} \}$ and $T_{a} < T_{b}$. Note that $T_{a}$ and $T_{b}$ are conjugates of each other and  we can take $f(T_{b}, T_{a})=T_{a}T_{b}$. Note that $\Psi:\mathcal{M}_{1} \to \text{SL}(2, \mathbb{Z})$ given by 
$$\Psi(T_{a})=
\begin{bmatrix}
1 & 1\\
0 & 1
\end{bmatrix}
\quad \textrm{and } \quad
\Psi(T_{b})=
\begin{bmatrix}
1 & 0\\
-1 & 1
\end{bmatrix}$$ is an isomorphism of groups. An elementary check gives $\C_{\text{SL}(2,\mathbb{Z})}(\Psi(T_a))=\langle \Psi(T_a),-\text{Id}\rangle$ and $-\text{Id}=(\Psi(T_a)\Psi(T_b))^3$. Thus, we have  $\C_{\mathcal{M}_{1}}(T_{a})=\langle T_{a}, T_{b}T_{a}T_{a}T_{b} \rangle$. By Theorem \ref{presentation-dehn-quandle}, a presentation of $\mathcal{D}_1$ has $S$ as its generating set and defining relations as follows:
\begin{eqnarray}
 T_{a}*T_{a}*T_{b}*T_{a}*^{-1}T_b *^{-1}T_a *^{-1}T_b & = & T_{a}, \label{m1e1} \\
  T_{b}*T_{a}*T_{b}*T_{a}*^{-1}T_b *^{-1}T_a *^{-1}T_b & = & T_{b},  \label{m1e2}\\
T_{a}*T_{b}*T_{a}*T_{b}*T_{a}*T_{b}*T_{a}*T_{b}*T_{a}*T_{b}*T_{a}*T_{b}*T_{a} &= &T_{a}, \label{m1e3}\\
T_{b}*T_{b}*T_{a}*T_{b}*T_{a}*T_{b}*T_{a}*T_{b}*T_{a}*T_{b}*T_{a}*T_{b}*T_{a} &= &T_{b},\label{m1e4}\\
T_{a}*T_{b}*T_{a}*T_{a}*T_{b}& = & T_{a}, \label{m1e5}\\
T_{a}*T_{b}*T_{a} & = & T_{b}. \label{m1e6}
\end{eqnarray}
Using the relation \eqref{m1e6}, we can write the relation~\eqref{m1e5} as $T_{b}*T_{a}*T_{b} =  T_{a}$. Now, using this new relation together with \eqref{m1e6}, one can recover relations \eqref{m1e1} through \eqref{m1e4}. Thus, we obtain the presentation
\begin{equation}\label{presentation of genus one quandle}
\mathcal{D}_1^{ns} \cong \langle T_{a}, T_{b}~\mid~ T_{a}*T_{b}*T_{a}=T_{b},  \quad T_{b}*T_{a}*T_{b}=T_{a} \rangle.
\end{equation}
\end{example}

\begin{remark}
Notice that, the right hand side of \eqref{presentation of genus one quandle} is the presentation of the knot quandle of the trefoil. This recovers the main result of \cite[Theorem 3.1]{NiebrzydowskiPrzytycki2009}.
\end{remark}

\begin{example}
Next, we give a presentation of the Dehn quandle $\dq_{2}^{ns}$ of the closed orientable surface $S_2$ of genus two. To simplify the notation, we set $\alpha_1:=T_{a_1}$, $\alpha_2:=T_{b_1}$, $\alpha_3:=T_{c_1}$, $\alpha_4:=T_{b_2}$ and $\alpha_5:=T_{a_2}$, where the curves are as in Figure~\ref{Humphries-curves}(B).  Then, by \cite[Section 4]{MargalitWinarski}, a presentation of the mapping class group $\mcg_2$ of $S_2$ is
\begin{eqnarray}\label{presentation of mcg_2}
\nonumber \mcg_2 &=& \Bigg\langle \alpha_1,\alpha_2,\alpha_3,\alpha_4,\alpha_5 \mid [\alpha_i,\alpha_j]=1 \text{ for } |i-j|\geq2, \quad \alpha_i\alpha_{i+1}\alpha_i=\alpha_{i+1}\alpha_i\alpha_{i+1} \text{ for } 1 \le i \le 4,\\
 && (\alpha_1\alpha_2\alpha_3\alpha_4\alpha_5)^6=1,\quad (\alpha_1\alpha_2\alpha_3\alpha_4\alpha_5\alpha_5\alpha_4\alpha_3\alpha_2\alpha_1)^2=1,\\
\nonumber  && [\alpha_1, (\alpha_1\alpha_2\alpha_3\alpha_4\alpha_5\alpha_5\alpha_4\alpha_3\alpha_2\alpha_1)]=1 \Bigg\rangle.
\end{eqnarray}

Let $S=\{\alpha_1,\alpha_2,\alpha_3,\alpha_4,\alpha_5\}$ and consider the total order $\alpha_1<\alpha_2<\alpha_3<\alpha_4<\alpha_5$ on $S$. Note that all Dehn twists along non-separating simple closed curves are conjugate in $\mcg_2$. We choose $f$ such that $f(\alpha_2,\alpha_1)=\alpha_1\alpha_2$, $f(\alpha_3,\alpha_2)=\alpha_2\alpha_3$, $f(\alpha_4,\alpha_3)=\alpha_3\alpha_4$, and $f(\alpha_5,\alpha_4)=\alpha_4\alpha_5$. Using \cite[Theorem 1]{MR1851559} we get a generating set for the mapping class group of a genus one surface with two boundary components. Then embedding this surface into a closed surface of genus two and using \cite[Remark 1.19]{Putman} gives
$$\C_{\mathcal{M}_{2}}(\alpha_1)=\langle \iota, \alpha_1,\alpha_3,\alpha_4,\alpha_5 \rangle,$$
where $\iota$ is the hyperelliptic involution. Note that $\iota=\alpha_1\alpha_2\alpha_3\alpha_4\alpha_5\alpha_5\alpha_4\alpha_3\alpha_2\alpha_1=\alpha_5\alpha_4\alpha_3\alpha_2\alpha_1\alpha_1\alpha_2\alpha_3\alpha_4\alpha_5$ (see \cite[Section 4]{MargalitWinarski}). Thus, by Theorem~\ref{presentation-dehn-quandle},  $\dq_{2}^{ns}$ is generated by $S$ and has defining relations as follows. Relations of type (1) are 
\begin{eqnarray}
 \alpha_i*\alpha_j*\alpha_{j+1}*\alpha_j*^{-1}\alpha_{j+1}*^{-1}\alpha_j*^{-1}\alpha_{j+1} &=& \alpha_i \quad \text{for}~ 1\leq i\leq 5 \label{m2e1}\\\
\nonumber  && \quad \quad  \text{and}~1\leq j\leq 4, \\
\alpha_i*\alpha_1*\alpha_j*^{-1}\alpha_1*^{-1}\alpha_j &=& \alpha_i \quad \text{for}~ 1\leq i\leq 5 \label{m2e5}\\
\nonumber &&\quad \quad \textrm{and}~ j=3, 4, 5,\\
\alpha_i*\alpha_2*\alpha_j*^{-1}\alpha_2*^{-1}\alpha_j &=& \alpha_i \quad \text{for}~ 1\leq i\leq 5, \label{m2e8}\\
\nonumber &&\quad \quad \textrm{and}~ j=4, 5,\\
\alpha_i*\alpha_3*\alpha_5*^{-1}\alpha_3*^{-1}\alpha_5 &=& \alpha_i \quad \text{for}~ 1\leq i\leq 5, \label{m2e10}\\
\alpha_i *\underbrace{\alpha_5*\alpha_4*\alpha_3*\alpha_2*\alpha_1* \cdots * \alpha_5*\alpha_4*\alpha_3*\alpha_2*\alpha_1}_{6 \text{ times}} &=& \alpha_i \quad \text{for}~1\leq i\leq 5, \label{m2e11}\\
\alpha_i*\alpha_1*\alpha_2*\alpha_3*\alpha_4*\alpha_5 *\alpha_5 *\alpha_4*\alpha_3*\alpha_2 &=& \alpha_i \quad \text{for}~ 1\leq i\leq 5,  \quad \label{m2e12}\\
\nonumber  *\alpha_1*\alpha_1*\alpha_2*\alpha_3*\alpha_4*\alpha_5*\alpha_5*\alpha_4*\alpha_3*\alpha_2*\alpha_1  &&\\
\alpha_i*\alpha_1*\alpha_2*\alpha_3*\alpha_4*\alpha_5*\alpha_5*\alpha_4*\alpha_3*\alpha_2*\alpha_1*\alpha_1   &=& \alpha_i \quad \textrm{for}~   1\leq i\leq 5. \label{m2e13}\\
 *^{-1}\alpha_1*^{-1}\alpha_2*^{-1}\alpha_3*^{-1}\alpha_4*^{-1}\alpha_5*^{-1 }\alpha_5*^{-1}\alpha_4*^{-1}\alpha_3*^{-1}\alpha_2*^{-1}\alpha_1*^{-1}\alpha_1
\nonumber &&
\end{eqnarray}
Relations of type (2) are
\begin{eqnarray}
\alpha_1*\alpha_5*\alpha_4*\alpha_3*\alpha_2*\alpha_1*\alpha_1*\alpha_2*\alpha_3*\alpha_4*\alpha_5 &=& \alpha_1, \label{m2e14}\\
\alpha_1*\alpha_i &=& \alpha_1 \quad \text{for}~ i=3, 4, 5,\label{m2e15}
\end{eqnarray}
and relations of type (3) are
\begin{eqnarray}
\alpha_i*\alpha_{i+1}*\alpha_i &=& \alpha_{i+1} \quad \text{for}~1\leq i\leq 4. \label{m2e18}
\end{eqnarray}

Note that relations \eqref{m2e1} can be recovered using relations \eqref{m2e18}.   Along similar lines, we can recover relations \eqref{m2e5} using relations \eqref{m2e15}. We can rewrite the relation~\eqref{m2e14} as $\alpha_2*\alpha_1*\alpha_2=\alpha_1$ using relations \eqref{m2e15}. Similarly, we can rewrite relations~\eqref{m2e8} and~\eqref{m2e10} as $\alpha_2*\alpha_4=\alpha_2, \alpha_2*\alpha_5=\alpha_2$ and $\alpha_3*\alpha_5=\alpha_3$, respectively, using relations \eqref{m2e15} and \eqref{m2e18}. One can see that relations ~\eqref{m2e11} can  be recovered using relations \eqref{m2e14}, \eqref{m2e15}, \eqref{m2e18} and relations $\alpha_2*\alpha_4=\alpha_2, \alpha_2*\alpha_5=\alpha_2$ and $\alpha_3*\alpha_5=\alpha_3$. For $i=2$, the relation \eqref{m2e13} can be rewritten as $\alpha_1*\alpha_2*\alpha_3*\alpha_4*\alpha_5*\alpha_5*\alpha_4*\alpha_3*\alpha_2=\alpha_1$, which further gives relations \eqref{m2e12}. Thus, $\dq_{2}^{ns}$ has a presentation
\begin{eqnarray}\label{presentation of d2ns}
\nonumber \dq_{2}^{ns} &=& \Bigg\langle \alpha_1, \alpha_2, \alpha_3,\alpha_4, \alpha_5~\mid~\quad \alpha_1*\alpha_3=\alpha_1, \quad \alpha_1*\alpha_4=\alpha_1, \quad \alpha_1*\alpha_5=\alpha_1,\\
 &&   \alpha_2*\alpha_4=\alpha_2, \quad \alpha_2*\alpha_5=\alpha_2,\quad  \alpha_3*\alpha_5=\alpha_3,  \quad \alpha_1*\alpha_2*\alpha_1=\alpha_2,\\
\nonumber && \alpha_2*\alpha_1*\alpha_2=\alpha_1, \quad \alpha_2*\alpha_3*\alpha_2=\alpha_3, \quad \alpha_3*\alpha_4*\alpha_3=\alpha_4, \quad \alpha_4*\alpha_5*\alpha_4=\alpha_5, \\
\nonumber && \alpha_1*\alpha_2*\alpha_3*\alpha_4*\alpha_5*\alpha_5*\alpha_4*\alpha_3*\alpha_2=\alpha_1 \Bigg\rangle.
\end{eqnarray}
\end{example}
\medskip

\subsection{An alternate approach to presentation of $\dq_{2}^{ns}$}
We conclude by giving an alternate proof of the presentation of $\dq_{2}^{ns}$ by identifying it with $\dq_{0,6}^{ns}$. It follows from \cite[Section 5.1.3]{Farb-Margalit2012} that a presentation of the mapping class group of $S_{0,6}$ is 
\begin{eqnarray*}
\mcg_{0,6} &=& \Bigg\langle \sigma_1,\sigma_2,\sigma_3,\sigma_4,\sigma_5 \mid \sigma_i\sigma_j=\sigma_j\sigma_i~ \text{for}~|i-j|\ge 2,\quad \sigma_i\sigma_{i+1}\sigma_i=\sigma_{i+1}\sigma_i\sigma_{i+1}~ \text{for all}~i, \\
&& \ (\sigma_1\sigma_2\sigma_3 \sigma_4\sigma_5)^6=1, \quad \sigma_1\sigma_2\sigma_3\sigma_4\sigma_5\sigma_5\sigma_4\sigma_3 \sigma_2\sigma_1=1 \Bigg\rangle.
\end{eqnarray*}
Recall that, the enveloping group $\Env(X)$ of a quandle $(X, *)$ is defined as
$$\Env(X)= \langle e_x, ~x \in X \mid e_{x*y}= e_y e_x e_y^{-1}, ~x, y \in X \rangle.$$
The natural map $\eta: X \to \Env(X)$ given by $\eta(x)=e_x$ is a homomorphism of quandles when we view $\Env(X)$ with the conjugation quandle structure. A presentation of the Dehn quandle $\dq_{0,6}^{ns}$ of $S_{0,6}$ is given in \cite[Theorem 3.2] {kamadamatsumoto} as
\begin{eqnarray}\label{presentation d06}
\nonumber \dq_{0,6}^{ns} &=& \Bigg\langle \sigma_1, \sigma_2, \sigma_3, \sigma_4, \sigma_5 \mid \sigma_i* \sigma_j= \sigma_i~\textrm{for}~ |i-j| \ge 2,\quad \sigma_i* \sigma_j * \sigma_i= \sigma_j~\textrm{for}~ |i-j| = 1,\\
&& \sigma_1*\sigma_2*\sigma_3*\sigma_4*\sigma_5= \sigma_5*\sigma_4*\sigma_3*\sigma_2*\sigma_1 \Bigg\rangle. 
\end{eqnarray}
Further, a presentation of the enveloping group of $\dq_{0,6}^{ns}$ follows from \cite[Theorem 3.7] {kamadamatsumoto} (also from \cite[Theorem 5.1.7]{Winker1984}) and given as
\begin{eqnarray}\label{presentation Endd06}
\Env(\dq_{0,6}^{ns}) &=& \Bigg\langle e_{\sigma_1},e_{\sigma_2},e_{\sigma_3},e_{\sigma_4},e_{\sigma_5} \mid e_{\sigma_i}e_{\sigma_j}=e_{\sigma_j}e_{\sigma_i} \text{ for } |i-j|\ge 2,\\
\nonumber && e_{\sigma_i}e_{\sigma_{i+1}}e_{\sigma_i}=e_{\sigma_{i+1}}e_{\sigma_i}e_{\sigma_{i+1}}~\textrm{for all}~ i,\\
\nonumber && [e_{\sigma_j},e_{\sigma_1}e_{\sigma_2} e_{\sigma_3}e_{\sigma_4} e_{\sigma_5}e_{\sigma_5}e_{\sigma_4}e_{\sigma_3} e_{\sigma_2}e_{\sigma_1}]=1 \text{ for all }j \Bigg\rangle.
\end{eqnarray}
Note that $\dq_{0,6}^{ns}$ is referred as  the quandle of cords of $S_{0,6}$ in \cite{kamadamatsumoto}.  It is easy to see that commutativity of $e_{\sigma_1}e_{\sigma_2} e_{\sigma_3}e_{\sigma_4} e_{\sigma_5}e_{\sigma_5}e_{\sigma_4}e_{\sigma_3} e_{\sigma_2}e_{\sigma_1}$ with $e_{\sigma_j}$ for $2\leq j\leq 5$ follows from braid relations, and hence the last relation in \eqref{presentation Endd06} is needed only for $j=1$. The following theorem together with \eqref{presentation d06} gives a presentation of $\dq_2^{ns}$.

\begin{proposition}
$\dq_2^{ns} \cong \dq_{0,6}^{ns}$.
\end{proposition}

\begin{proof}
The map $\eta: \dq_{0,6}^{ns}\to \Env(\dq_{0,6}^{ns})$, given by $\eta(x)=e_x$, induces a quandle homomorphism $\eta': \dq_{0,6}^{ns}\to \dq(\eta(S)^{\Env(\dq_{0,6}^{ns})})$, where $S$ is the generating set of $\dq_{0,6}^{ns}$ as in \eqref{presentation d06}. By \cite[Section 3.2]{Dhanwani-Raundal-Singh-2021}, the map $$\Phi: \Env(\dq_{0,6}^{ns}) \to \mcg_{0,6}$$ given by $\Phi(e_x)=x$ is a surjective group homomorphism,  and hence induces a surjective quandle homomorphism $\Phi' : \dq(\eta(S)^{\Env(\dq_{0,6}^{ns})}) \to \dq_{0,6}^{ns}$. Since $\eta' \Phi'$ and $\Phi' \eta'$ are both identity maps, $\Phi'$ is a quandle isomorphism. Now, recall the presentation of $\mcg_2$ from \eqref{presentation of mcg_2}. It follows that we can factorise  $\Phi=\Phi_2\Phi_1$, where
$\Phi_1 : \Env(\dq_{0,6}^{ns}) \to \mcg_2$ and $\Phi_2 : \mcg_2 \to \mcg_{0,6}$ are surjective group homomorphisms. These maps induce surjective quandle homomorphisms $\Phi_1' : \dq(\eta(S)^{\Env(\dq_{0,6}^{ns})})\to \dq_2^{ns}$ and $\Phi_2' : \dq_2^{ns}\to \dq_{0,6}^{ns}$ such that $\Phi'=\Phi_2'\Phi_1'$. Since $\Phi'$ is an isomorphism, it follows that both $\Phi_2', \Phi_1'$ are isomorphisms, and hence $\dq_2^{ns} \cong \dq_{0,6}^{ns}$.
\end{proof}

\begin{remark}
Note that, the additional braid and far commutativity induced relations in \eqref{presentation d06} can be recovered using relations from \eqref{presentation of d2ns}. 
Further,  the long relation $\alpha_1*\alpha_2*\alpha_3*\alpha_4*\alpha_5*\alpha_5*\alpha_4*\alpha_3*\alpha_2=\alpha_1$ in \eqref{presentation of d2ns} is equivalent to the long relation 
$\sigma_1*\sigma_2*\sigma_3*\sigma_4*\sigma_5= \sigma_5*\sigma_4*\sigma_3*\sigma_2*\sigma_1$ in \eqref{presentation d06}. 
\end{remark}
\medskip

\begin{ack}
Neeraj K. Dhanwani thanks IISER Mohali for the institute post doctoral fellowship. Hitesh Raundal acknowledges financial support from the SERB grant SB/ SJF/2019-20. Mahender Singh is supported by Swarna Jayanti Fellowship grants DST/SJF/MSA-02/2018-19 and SB/SJF/2019-20.
\end{ack}

\end{document}